\documentclass{article}

\usepackage{geometry,amsmath,amsthm,amssymb,mathrsfs,graphicx,extarrows,float,mathabx}
\usepackage[numbers]{natbib}
\usepackage{titlesec}
\usepackage[titletoc,toc,title]{appendix}
\usepackage{graphicx}
\usepackage{faktor}
\usepackage{hyperref}
\usepackage{latexsym,fullpage,amsfonts,amssymb,amsmath,amscd,graphics,epic}
\usepackage[all]{xy}
\usepackage{amssymb,amsthm,amsxtra}
\usepackage[usenames]{color}
\usepackage{amscd}
\usepackage{amsthm}
\usepackage{amsfonts}
\usepackage{amssymb}
\usepackage{mathrsfs}
\usepackage{mathtools}
\usepackage{caption}
\usepackage{subcaption}

\newtheorem{thm}{Theorem}[section]
\newtheorem{prop}[thm]{Proposition}
\newtheorem{defin}[thm]{Definition}

\newtheorem{corr}[thm]{Corollary}
\newtheorem{lemma}[thm]{Lemma}
\newtheorem{remark}[thm]{Remark}
\newtheorem{exmp}[thm]{Example}

\newtheorem*{thm*}{Theorem}
\newtheorem*{defin*}{Definition}

\DeclareMathOperator{\oh}{\mathcal{O}_q(H(N))}
\DeclareMathOperator{\ohtwo}{\mathcal{O}_q(H(2))}
\DeclareMathOperator{\boh}{\mathfrak{o}_q(H(N))}
\DeclareMathOperator{\bohtwo}{\mathfrak{o}_q(H(2))}

\DeclareMathOperator{\ogl}{\mathcal{O}_q(GL(N))}
\DeclareMathOperator{\bogl}{\mathfrak{o}_q(GL(N))}

\DeclareMathOperator{\ugl}{U_q(\mathfrak{gl}_N)}
\DeclareMathOperator{\bgl}{\mathfrak{u}_q(\mathfrak{gl}_N)}
\DeclareMathOperator{\bgltwo}{\mathfrak{u}_q(\mathfrak{gl}_2)}
\DeclareMathOperator{\bglk}{\mathfrak{u}_q^{(k)}(\mathfrak{gl}_N)}
\DeclareMathOperator{\bglN}{\mathfrak{u}_q^{(N)}(\mathfrak{gl}_N)}
\DeclareMathOperator{\bglt}{\mathfrak{u}_q^{(2)}(\mathfrak{gl}_N)}
\DeclareMathOperator{\usl}{U_q(\mathfrak{sl}_N)}
\DeclareMathOperator{\usltwo}{U_q(\mathfrak{sl}_2)}
\DeclareMathOperator{\bsl}{\mathfrak{u}_q(\mathfrak{sl}_N)}
\DeclareMathOperator{\bsltwo}{\mathfrak{u}_q(\mathfrak{sl}_2)}
\DeclareMathOperator{\bsla}{\tilde{\mathfrak{u}}_q(\mathfrak{sl}_N)}
\DeclareMathOperator{\bslatwo}{\tilde{\mathfrak{u}}_q(\mathfrak{sl}_2)}
\DeclareMathOperator{\tp}{\tilde{p}}
\DeclareMathOperator{\hr}{\hat{R}}
\DeclareMathOperator{\mc}{\mathbb{C}}
\DeclareMathOperator{\mr}{\mathbb{R}}
\DeclareMathOperator{\mz}{\mathbb{Z}}
\DeclareMathOperator{\mn}{\mathbb{N}}
\DeclareMathOperator{\id}{id}
\DeclareMathOperator{\ot}{\mathcal{O}_q(T(N))}
\DeclareMathOperator{\otu}{\mathcal{O}_q(T^u(N))}
\DeclareMathOperator{\otl}{\mathcal{O}_q(T^l(N))}
\DeclareMathOperator{\bott}{\mathfrak{o}_q(T(N))}
\DeclareMathOperator{\botu}{\mathfrak{o}_q(T^u(N))}
\DeclareMathOperator{\botl}{\mathfrak{o}_q(T^l(N))}

\DeclareMathOperator{\botk}{\mathfrak{o}_q^{(k)}(T(N))}
\DeclareMathOperator{\br}{\mathbf{r}}
\DeclareMathOperator{\omn}{\mathcal{O}_q(M_N)}

\DeclareMathOperator{\mci}{\mathcal{I}}
\DeclareMathOperator{\hk}{\hat{K}}
\DeclareMathOperator{\Hom}{Hom}
\DeclareMathOperator{\Tr}{Tr}
\DeclareMathOperator{\Rep}{Rep}
\DeclareMathOperator{\mcs}{\mathcal{S}}
\DeclareMathOperator{\ote}{\mathcal{O}_{q,\epsilon}(T(N))}
\DeclareMathOperator{\bote}{\mathfrak{o}_{q,\epsilon}(T(N))}
\DeclareMathOperator{\botek}{\mathfrak{o}_{q,\epsilon}^{(k)}(T(N))}

\usepackage{mwe}

\newcommand{\Addresses}{{
		\bigskip
		\footnotesize 
		
		\textsc{Department of Pure Mathematics, Xi'an Jiaotong-Liverpool University, Suzhou, China}\par\nopagebreak
		\textit{E-mail address}: \texttt{Stephen.Moore@xjtlu.edu.cn} }}

\title{On a Small Version of the Reflection Equation Algebra}
\author{Stephen T. Moore}
\date{}

\begin{document}
	
	\maketitle 
	
\begin{abstract}
\noindent We give an alternative presentation of the small version of the reflection equation algebra associated to $GL_N$ at both odd and even roots of unity, and use our presentation to classify its irreducible representations. We then describe a family of algebras generalizing the small reflection equation algebra, and consider their application to the study of module categories over $\usl$ fusion categories.
\end{abstract}

\section{Introduction}
The \textit{reflection equation} was introduced by Cherednik \cite{Cherednik} in relation to quantum integrable systems with boundary condition. The \textit{reflection equation algebra} (REA) \cite{KulishSklyanin, Sklyanin} was introduced to study solutions of the reflection equation algebraically. More recently, the reflection equation algebra has found applications in areas such as non-commutative geometry and low-dimensional topology \cite{BZBJ, Majid2, Mudrov}. 

In general, the reflection equation algebra is not a Hopf algebra, but instead a comodule algebra over a Hopf algebra. The reflection equation depends on a choice of solution of the Yang-Baxter equation, and similarly the reflection equation algebra depends on such a choice. Starting with the $R$-matrix solution coming from the standard representation of $\ugl$, the associated reflection equation algebra $\oh$ forms a comodule algebra over both $\ugl$ and $\ogl$. This means that the representation category of $\oh$ can be viewed as a module category over both $\Rep\ugl$ and $\Rep\ogl$. 

A form of $\oh$ can actually be constructed from both $\ogl$ and $\ugl$. Firstly, $\oh$ can be obtained from $\ogl$ by modifying or twisting the product on $\ogl$ using its co-quasitriangular structure \cite{Majid}. Secondly, a form of $\oh$ can be obtained from $\ugl$ by requiring that certain special elements of $\oh$ have a formal inverse \cite{JosephLetzter, Rosso}. 

The algebras $\ugl$, $\ogl$, and $\oh$ all depend on a choice of deformation parameter $q\in\mc$. When $q$ is a root of unity, it was shown by Lusztig \cite{Lusztig} that $\ugl$ has a finite dimensional Hopf algebra quotient called the \textit{small quantum group} $\bgl$. It was later shown that $\ogl$ similarly has a finite dimensional or small quotient $\bogl$ \cite{ParshallWang, Takeuchi}.

As $\oh$ can be constructed from both $\ogl$ and $\ugl$, it is natural to ask if these constructions can be combined with the respective quotients at roots of unity, to produce a small version of $\oh$. This was considered in \cite{CookeLaugwitz}, where a presentation for the small REA $\boh$ at odd roots of unity was given by twisting the quotient relations that define $\bogl$. 

Our aim here is to reconsider this construction using certain special elements of $\ogl$ and $\oh$ called \textit{quantum minors}. This then allows us to give an alternative presentation for a small version of $\oh$ at both odd and even roots of unity. An immediate consequence of our alternative presentation is that it allows us to relate $\boh$ to the small quantum group $\bgl$:
\begin{thm*}
For $q$ a root of unity, there is an algebra isomorphism
\[
\boh\simeq\bglt.\]
\end{thm*}
Here the $(2)$ labelling $\bglt$ denotes that we have to be careful about the choice of certain relations in the definition of $\bgl$ at even roots of unity. The theorem also requires allowing square roots of certain elements in $\boh$ at even roots of unity. We note that at least for $N=2$, the form of small quantum group appearing in this theorem at even roots of unity is known to be quasitriangular (see Remark \ref{remark: quasitriangularity of Uq(sl2)} and Corollary \ref{corr: boh simeq bgl}). From this isomorphism, it is then possible to classify the irreducible representations of $\boh$.

For $q$ generic, the bounded irreducible $*$-representations of $\oh$ were classified in \cite{DCM1, DCM2, Me1}. A key component of this classification was the notion of \textit{shape invariant} associated to the REA, with every irreducible having a unique shape. More explicitly, each shape identifies a particular ideal of $\oh$, and irreducibles associated to a given shape factor through the corresponding quotient. With this in mind, it can be see that the small REA actually factors through a particular shape. It is then natural to ask if we can define similar quotients of the REA associated to other shapes. Motivated by this, we then introduce a family of algebras generalizing the small REA:
\begin{defin*}
For $q$ a root of unity, a \textbf{wee reflection equation algebra} is a finite dimensional quotient of $\oh$ that forms a comodule algebra over $\bgl$.
\end{defin*}
We focus on two families of examples of wee REAs. The first family are motivated by the \textit{big cell representations} of $\oh$ constructed in \cite{DCM1}, and can be thought of as natural generalizations of $\boh$. The second family are motivated by \textit{exact comodule algebras}. These were defined to produce exact module categories over representation categories of Hopf algebras. For $q$ an odd root of unity, the exact comodule algebras over $\bsltwo$ were classified in \cite{Mombelli}, and given explicit presentations in \cite{NSS}. We consider when these exact comodule algebras can be viewed as wee REAs.

By definition, the representation category of any wee REA will naturally form a module category of the representation category of $\bgl$. If we instead consider an appropriate representation category of $\bsl$, we can obtain a fusion category by \textit{semisimplification}. By considering a compatible notion of semisimplification for the representation category of a wee REA, we can obtain a module category over a $\usl$ fusion category. We demonstrate this by constructing the $\usltwo$ module categories $T_n$ and $D_n$ from representations of $\ohtwo$.

\subsection*{Contents}
In Section \ref{section: 1} we review the necessary background on the quantum groups $\ogl$ and $\ugl$, along with their small versions. We similarly review the necessary details of the reflection equation algebra in Section \ref{section: 2}, along with how it can be constructed from the quantum groups $\ogl$ and $\ugl$. The alternative presentation for $\boh$ is given in Section \ref{section: 3}, as well as the classification of its irreducible representations. In Section \ref{section: 4} we define wee REAs, and consider two families of examples, namely wee REAs of \textit{big cell type}, and wee REAs coming from \textit{exact comodule algebras}. Finally in Section \ref{section: 5} we consider semisimplification of module categories, and show how to construct the $\Rep\usltwo$ modules categories $T_n$ and $D_n$ by semisimplification.

\subsection*{Notation}
For $q\in\mc\setminus\{\pm 1\}$ a root of unity, we denote by $p\in\mn$ the smallest number such that $q^p=1$. We denote by $\tp\in\mn$ the smallest number such that $q^{\tp}=\pm 1$. If $q^{\tp}=-1$, we call $q$ an \textit{even root of unity}, and if $q^{\tp}=1$, we call $q$ an \textit{odd root of unity}.

Let $A$ and $B$ be totally ordered sets, and let $\sigma:A\rightarrow B$ be a bijection. We denote by
\[
l(\sigma):=\{(i,j)\in A\times A \lvert i<j, \sigma(i)>\sigma(j)\}\]
the number of inversions of $\sigma$. We will often denote certain sets as follows:
\[
[k]:=\{1,2,...,k\}.\]

\section{Background}\label{section: 1}

\subsection{The $\hr$-matrix}
For $q\in\mc$ and $N\in\mn$ we define the $N\times N$ matrix
\[
\hr:= \sum\limits_{i,j} q^{-\delta_{ij}}e_{ji}\otimes e_{ij} + (q^{-1}-q)\sum\limits_{i<j}e_{jj}\otimes e_{ii}.\]
This operator satisfies the \textit{braid relation}
\[
\hr_{12}\hr_{23}\hr_{12} = \hr_{23}\hr_{12}\hr_{23},\]
where $\hr_{ij}$ denotes $\hr$ acting on the $i$th and $j$th copies of $\mc^N\otimes \mc^N\otimes \mc^N$.

It has inverse given by
\[
\hr^{-1} = \sum\limits_{i,j}q^{\delta_{ij}}e_{ij}\otimes e_{ji}+(q-q^{-1})\sum\limits_{i<j}e_{ii}\otimes e_{jj}.\]

\subsection{ $\omn$, $\ogl$ and $\bogl$}
We generally follow \cite{KlimykSchmudgen} or \cite{ParshallWang}.

Let $X$ denote the $N\times N$ matrix with variables $(x_{ij})_{1\leq i,j\leq N}$. $\omn$ is the unital bialgebra generated by the elements $x_{ij}$, $1\leq i,j\leq N$, with relations given by
\[
\hr_{12}X_{13}X_{23} = X_{13}X_{23}\hr_{12}.\]
The coproduct and counit on $\omn$ is given by
\[
(\id\otimes\Delta)X = X_{12}X_{13}, \qquad (\id\otimes\varepsilon)X = I_N.\]
$\omn$ contains a unique (group-like) central element $D_q\in\omn$, called the \textit{quantum determinant}, which is given by
\[
D_q:= \sum\limits_{\sigma\in S_N}(-q)^{l(\sigma)}x_{1,\sigma(1)}...x_{N,\sigma(N)}\]
If we require the quantum determinant to have a formal inverse $D_q^{-1}$, then the matrix $X$ becomes invertible, and $\omn[D_q^{-1}]$ will become a Hopf algebra, which we denote by $\ogl$. The antipode of $\ogl$ is given by 
\[
S(x_{ij}):=(X^{-1})_{ij}.\]
To describe this more explicitly, we need the following:
\begin{defin}
Given $I,J\subseteq[N]$, $I=\{i_1,...,i_k\}$, $J=\{j_1,...,j_k\}$, where the elements of the sets are ordered from smallest to largest, we define the \textbf{quantum minor} $X_{IJ}\in\ogl$ by
\begin{gather}\label{eq: quantum minor expansion}
X_{\{i_1,i_2\},\{j_1,j_2\}} := x_{i_1 j_1}x_{i_2j_2}-qx_{i_1j_2}x_{i_2j_1}, \\	
X_{IJ} := \sum\limits_{1\leq l\leq k}(-q)^{l-1}x_{i_1,j_l}X_{\{i_2,...,i_k\},J\setminus\{j_l\}}.
\end{gather}
\end{defin}
The quantum determinant is then just a particular quantum minor, $D_q=X_{[N],[N]}$. We note that quantum minors are defined using an expansion formula in terms of lower rank quantum minors. There are in fact multiple ways to expand a quantum minor in terms of lower rank ones, and a more general formula for the expansion can be seen for example in \cite{ParshallWang}[Theorem 4.4.3] (or \cite{DCM1}[Proposition 2.8] for our notation).

The quantum minors of a rixed rank form the basis of an irreducible $\omn$ comodule. The coproduct on quantum minors is given by
\begin{equation}
\Delta(X_{IJ}) = \sum\limits_{\substack{K\subseteq[N] \\ \lvert K\rvert = \lvert I\rvert}}X_{IK}\otimes X_{KJ}.
\end{equation}
Explicitly, the antipode on $\ogl$ is given in terms of the quantum minors by
\begin{gather}\label{eqn: ogl antipode}
S(x_{ij})=(-q)^{i-j}X_{[N]\setminus\{j\},[N]\setminus\{i\}}D_q^{-1}.
\end{gather}


\subsubsection{Small $\ogl$}
A general construction of $\bogl$ for odd roots of unity was given in \cite{ParshallWang}[Section 7], and a construction for even roots of unity follows similarly:

At roots of unity, certain elements in $\ogl$ become central, and the small quantum group $\bogl$ is then a Hopf algebra quotient by these central elements. To describe this, we first need to consider general commutation relations in $\bogl$. For $x_{ij},x_{kl}\in\ogl$, if $ (i,j)=(l,k)$, or more generally $i<k$ and $j>l$, then
\[
x_{ij}x_{kl}=x_{kl}x_{ij}.\]
For the cases where the indices do not satisfy the above conditions, the general form of the commutation relation is given by
\begin{equation}\label{eqn: ogl commutation relations}
\begin{split}
q^{n\delta_{jl}(\delta_{i>k}-\delta_{k>i})+n\delta_{ik}(\delta_{j>l}-\delta_{l>j})}x_{ij}^nx_{kl}-x_{kl}x_{ij}^n = (q-q^{1-2n})\left(\delta_{k>i}\delta_{l>j}-q^{2n-2}\delta_{i>k}\delta_{j>l}\right)x_{ij}^{n-1}x_{il}x_{kj},
\end{split}
\end{equation} 
for a choice of $n\in\mn$. From this, we can see that for odd roots of unity, the elements $x_{ij}^{p}$ will be central in $\ogl$. 
\begin{defin}
	For odd roots of unity, the \textbf{small quantum group} $\bogl$ is defined as the quotient of $\ogl$ by the relations
	\[
	x_{ij}^p = \delta_{ij}.\]
\end{defin}
For even roots of unity, the situation is slightly more complicated. In this case, taking $n= \tp$, we see that the elements $x_{ij}^{\tp}$ will in general anti-commute in $\ogl$. Hence, to ensure the quotient is compatible with this, we instead define the small quantum group as follows:
\begin{defin}\label{def: bogl}
For even roots of unity, the \textbf{small quantum group} $\bogl$ is defined as the quotient of $\ogl$ by the relations
\[
x_{ii}^{p}=1, \qquad x_{ij}^{\tp}=0 \text{ for }i\neq j.\]
\end{defin}
For the Hopf algebra structure, \cite{ParshallWang}[Lemma 7.2.2] generalizes similarly for even roots of unity as follows:
\[
\Delta(x_{ij}^{\tp}) = \sum\limits_{r=1}^{N}x_{ir}^{\tp}\otimes x_{rj}^{\tp}, \qquad \varepsilon(x_{ij}^{\tp}) = \delta_{ij},\] 
so the ideal defined by the relations in Definition \ref{def: bogl} can be seen to form a bi-ideal. To check that it is in fact a Hopf ideal requires the following identity from \cite{ParshallWang}[Lemma 7.2.3]. We note that again their proof can be seen to also hold at even roots of unity:
\begin{prop}\label{prop: power of quantum determinant, minor expansion}
In $\omn$ for $q$ a root of unity, the quantum determinant satisfies
\[
D_q^{\tp} = \sum\limits_{i=1}^{N}(\lambda)^{i-1}x_{Ni}^{\tp}X_{[N-1],[N]\setminus\{i\}}^{\tp},\]
where $\lambda=1$ if $q$ is an even root of unity and $\tp$ is odd, and $\lambda=-1$ otherwise.
\end{prop}

A consequence of Proposition \ref{prop: power of quantum determinant, minor expansion} is the following:
\begin{prop}\label{prop:power of small quantum determinant}
In $\bogl$, the quantum determinant satisfies $D_q^{p}=1$.
\end{prop}

Combining Proposition \ref{prop: power of quantum determinant, minor expansion} with Equation \ref{eqn: ogl antipode}, we now see that the bi-ideal defined by the relations in Definition \ref{def: bogl} is closed under the antipode, and therefore a Hopf ideal. Hence for all roots of unity $\bogl$ is a finite dimensional Hopf algebra. Further, for $q$ an odd root of unity, $\bogl$ has dimension $p^{N^2}$, whilst for even roots of unity it has dimension $2^N\tp^{N^2}$.

\subsection{$\ot$ and $\bott$}
Let $T^u$ denote the matrix of variables $(t_{ij})_{1\leq i\leq j\leq N}$. Let $T^l$ denote the matrix with variables $(t_{ij})_{1\leq j\leq i\leq N}$. The algebra $\ot$ is generated by the elements $t_{ij}$, $1\leq i,j\leq N$, along with formal inverses $t_{ii}^{-1}$, with relations given by
\[
\hr_{12}T^u_{13}T^u_{23} = T^u_{13}T^u_{23}\hr_{12}, \qquad \hr_{12}T^l_{23}T^l_{13} = T^l_{23}T^l_{13}\hr_{12}, \qquad T^u_{23}\hr_{12}T^l_{23} = T^l_{13}\hr_{12}T^u_{13}.\]
$\ot$ becomes a Hopf algebra with coproduct and counit given by 
\[
(\id\otimes\Delta)T^u = T^u_{12}T^u_{13}, \quad (\id\otimes\Delta)T^l = T^l_{13}T^l_{12}, \quad (\id\otimes\varepsilon)T^u = I_N, \quad (\id\otimes\varepsilon)T^l = I_N.\]
We denote by $\otu$ the Hopf subalgebra of $\ot$ generated by $T^u$, i.e. the elements $t_{ij}$ with $1\leq i\leq j\leq N$, and the inverses $t_{ii}^{-1}$. $\otl$ is defined similarly.
$\otu$ can be viewed as a quotient Hopf algebra of $\ogl$ given by
\[
x_{ij}\mapsto t_{ij} \text{ for }i\leq j, \quad x_{kl}\mapsto 0 \text{  for  } k>l,\] 
and it can be seen to inherit its Hopf algebra structure from this quotient. For $q$ generic, the relations on $\otl$ can be obtained by applying the $*$-structure 
\[
t_{ij}^* = t_{ji},\]
to $\otu$. 

\subsubsection{Small $\ot$}

From relation \ref{eqn: ogl commutation relations}, we can see that the quotients 
\[
\ogl\rightarrow\bogl ~ \text{ and } ~ \ogl\rightarrow\otu \]
are compatible, and hence we can define finite dimensional quotient Hopf algebras $\botu$ and $\botl$ by
\[
t_{ii}^{p}=1, \qquad t_{ij}^{\tp}=0, ~ i\neq j.\]
We now want to verify that we can extend this to a quotient Hopf algebra structure on $\ot$. For $i<j$, $k>l$, and $n\in\mn$, the general form of commutation relation is given by
\begin{equation}
	\begin{split}
t_{ij}^nt_{kl}-q^{n(\delta_{jk}-\delta_{il})}t_{kl}t_{ij}^n =& \delta_{il}q^{\delta_{jk}-1}(q^{2-2n}-q^2)t_{ij}^{n-1}\left(\sum\limits_{m=i+1}^{\min\{j,k\}}t_{km}t_{mj}\right) \\
+& \delta_{jk}(q^{2n}-1)t_{ij}^{n-1}\left(\sum\limits_{m=\max\{i,l\}}^{j-1}t_{im}t_{ml}\right) \\
+& \delta_{il}\delta_{jk}(q^{2n}-q^2-1+q^{2-2n})t_{ij}^{n-2}\left(\sum\limits_{i+1=s\leq m}^{j-1}t_{im}t_{ms}t_{sj}\right),
\end{split} 
\end{equation}
where we note we can obtain the relation for $i>j$ and $k<l$ by applying the $*$-structure.

From this, we see that again the elements $t_{ij}^{\tp}$ will either commute or anti-commute, and hence the corresponding quotient of $\ot$ will form a finite dimensional Hopf algebra.
\begin{defin}
We define $\bott$ to be the finite dimensional Hopf algebra obtained by quotienting $\ot$ by the relations
\[
t_{ii}^p=1, \qquad t_{ij}^{\tp}=0 ~ \text{ for } ~ i\neq j.\]
\end{defin}
Under the the quotient map $\ogl\rightarrow\otu$, $x_{ij}\mapsto\delta_{i\leq j}t_{ij}$, in terms of the quantum determinant, we have
\[
D_q\mapsto\prod\limits_{i=1}^{N}t_{ii}.\]
The image of $D_q$ is again central and group-like in $\ot$.

For even roots of unity, we will sometimes need to consider larger quotient algebras, so we introduce the following generalization:
\begin{defin}\label{def: botk}
For $q$ an even root of unity and $k\in\mn$, we define $\botk$ to be the finite dimensional Hopf algebra quotient of $\ot$ defined by
\[
t_{ii}^{kp}=1, \qquad t_{ij}^{\tp}=0 ~ \text{ for } ~ i\neq j.\]
\end{defin}

\subsection{$\ugl$, $\usl$ $\bgl$, and $\bsl$}
The Hopf algebra $\ugl$ in generated by the elements $E_j,F_j,K_i$, $1\leq j\leq N-1$, $1\leq i\leq N$, with relations
\[
K_iE_j = q^{\delta_{ij}-\delta_{i,j+1}}E_jK_i, \quad K_iF_j = q^{\delta_{i,j+1}-\delta_{ij}}F_jK_i, \quad E_jF_k - F_k E_j = \delta_{jk}\frac{K_jK_{j+1}^{-1}-K_{j+1}K_j^{-1}}{q-q^{-1}},\]
\[
E_i^2E_{i\pm 1}-(q+q^{-1})E_iE_{i\pm 1}E_i+E_{i\pm 1}E_i^2 = 0, \quad F_i^2F_{i\pm 1}-(q+q^{-1})F_iF_{i\pm 1}F_i+F_{i\pm 1}F_i^2 = 0,\]
\[
E_jE_k=E_kE_j, \quad F_jF_k = F_kF_j\quad \text{if} \quad \lvert j-k\rvert >1.\]
The Hopf algebra structure is given by
\[
\Delta(K_i) = K_i\otimes K_i, \quad \Delta(E_j) = 1\otimes E_j+E_j\otimes K_jK_{j+1}^{-1}, \quad \Delta(F_j) = F_j\otimes 1+K_{j+1}K_j^{-1}\otimes F_j,\]
\[
\varepsilon(K_i)=1, \quad \varepsilon(E_j) = \varepsilon(F_j)=0,\]
\[
S(K_i)=K_i^{-1}, \quad S(E_j)=-E_jK_{j+1}K_j^{-1}, \quad S(F_j) = -K_jK_{j+1}^{-1}F_j.\]
The Hopf algebra $\usl$ is the sub-algebra of $\ugl$ generated by $E_i,F_i$, $\hat{K}_i:=K_iK_{i+1}^{-1}$, $1\leq i\leq N-1$.

For the following, see for example \cite{KlimykSchmudgen}[Theorem 8.33]:
\begin{thm}\label{thm: ugl simeq ot}
There is a Hopf algebra isomorphism $\ugl^{\text{cop}}\simeq\ot$, with
\[
t_{ii}\mapsto K_{i}^{-1}, \qquad t_{i+1,i}\mapsto (1-q^2)K_i^{-1}E_i, \qquad t_{i,i+1}\mapsto(q^{-1}-q)F_iK_{i+1}^{-1}.\]
\end{thm}

\subsubsection{Small $\ugl$ and $\usl$}

\begin{defin}
For $q$ a root of unity, the \textbf{small quantum group} $\bgl$ is the finite dimensional Hopf algebra quotient of $\ugl$ given by 
\[
E_{\alpha}^{\tp}=F_{\alpha}^{\tp}=0, \qquad K_j^p=1,\]
where $\alpha\in\Phi^{+}$ is a positive root. The \textbf{small quantum group} $\bsl$ is the finite dimensional Hopf algebra quotient of $\usl$ by the relations
\[
E_{\alpha}^{\tp}=F_{\alpha}^{\tp}=0, \qquad \hk_i^p=1.\]
\end{defin}
We note that for even roots of unity, the small quantum group $\bsl$ is sometimes referred to as the \textit{restricted quantum group} \cite{FGST}. For simplicity, we refer to both cases as the small quantum group. 

More generally, for $k\in\mn$, $\bglk$ is the Hopf algebra quotient of $\ugl$ defined by
\[
K_{j}^{kp}=1, \qquad E_{\alpha}^{\tp}=F_{\alpha}^{\tp}=0.\]	

The isomorphism in Theorem \ref{thm: ugl simeq ot} can be extended to the respective small algebras:
\begin{prop}
There is an isomorphism of Hopf algebras
\[
\botk\simeq\bglk^{\text{cop}}.\]
\end{prop} 
\begin{proof}
This can be seen by comparing the defining quotient relations on each side. In particular the relation
\[
t_{ii}^{kp}=1 ~ \text{ is equivalent to } ~ K_i^{kp}=1,\]
whilst the relation
\[
t_{i+1,i}^{\tp}=0 ~ \text{ is equivalent to } ~ E_i^{\tp}=0,\]
with $t_{i+j,i}^{\tp}=0$, $j>1$, corresponding to the other positive root relations.
Finally, using the fact that $\dim\bsl=p^{N-1}\tp^{\text{rank}({\mathfrak{g}})-N+1}$, we can see that both algebras have dimension $p^{kN}\tp^{N^2-N}$.
\end{proof}

For our purposes, we want to obtain a copy of $\Rep \bsl$ as a subcategory of $\Rep\bglk$, so that $\Rep\bsl$ can act on $\Rep \oh$. For this, we need to instead view $\bsl$ as a quotient of $\bglk$. Motivated by \cite{KlimykSchmudgen}[Section 8.5.3], we define the following:
\begin{defin}\label{def: alt bsla}
For $q$ an odd root of unity, we define $\bsla$ to be the Hopf algebra quotient of $\bgl$ by the	relation
\[
\prod\limits_{i=1}^{N}K_i=1.\]
For $q$ an even root of unity, we instead define $\bsla$ to be the Hopf algebra quotient of $\bglN$ by the same relation.
\end{defin}
\begin{remark}\label{remark: quasitriangularity of Uq(sl2)}
Considering the $N=2$ case, the above relation combined with the $\usltwo$ generator $\hk:=K_1K_2^{-1}$ gives
\[
\hk=K_1^2.\]
Then for $q$ an even root of unity, if we require $\hk^p=1$, it in turn requires $K_1^{2p}=1$, so $\bslatwo$ will contain a square root of $\hk$. In fact $\bslatwo$ appeared previously in \cite{FGST} (see also \cite{KondoSaito}[Section 4]), where it was noted that $\bslatwo$ is a quasitriangular Hopf algebra that contains $\bsltwo$ as a sub-algebra. 

Alternatively, for $q$ an odd root of unity, if 
\[
K^p_1=1,\]
then 
\[
\hk^{\frac{p+1}{2}}=K^{p+1}_1=K_1,\]
and so requiring higher orders of $K_1$ is not necessary.
\end{remark}

\section{The Reflection Equation Algebra, $\oh$}\label{section: 2}
Let $Z$ denote the $N\times N$ matrix of variables $(z_{ij})_{1\leq i,j \leq N}$. The \textit{reflection equation algebra}, $\oh$, is the unital algebra generated by the variables $z_{ij}$, $1\leq i,j\leq N$ with relations given by
\[
\hr_{12}Z_{23}\hr_{12}Z_{23} = Z_{23}\hr_{12}Z_{23}\hr_{12}.\]
In \cite{DCM1,DCM2,Me1}, the algebra $\oh$ was studied with added $*$-structure given by $Z^{*}=Z$, resulting in the classification of its bounded highest weight $*$-representations. With this added $*$-structure, $\oh$ can be considered as a deformation of the algebra of functions on the space of $N\times N$ hermitian matrices, hence the notation. In the current paper, we no longer consider this added $*$-structure, however we will keep the notation $\oh$ to refer to the reflection equation algebra.

The centre of $\oh$ is a polynomial algebra with generators $\sigma_k$, $1\leq k\leq N$. An explicit formula for these was given in \cite{JordanWhite}, see also \cite{Flore}. The elements $\sigma_1$ and $\sigma_N$ are often referred to as the \textit{quantum trace} and \textit{quantum determinant} respectively.

There is natural embedding of $\mathcal{O}_q(H(N-1))$ into $\oh$, just by mapping $z_{ij}$ to itself. Denoting $\sigma_k^{(k)}$ for the quantum determinant in $\mathcal{O}_q(H(k))$, then each $\sigma_k^{(k)}$ will be a $q$-commuting element when viewed as an element of $\oh$ for $N\geq k$, and are referred to as the \textit{leading quantum minors}.  

If we add the extra conditions that the leading quantum minors $\sigma_k^{(k)}$ are all invertible, and have a square root, for $1\leq k\leq N$, then we can obtain the following \cite{JosephLetzter, Rosso, JordanWhite}:
\begin{thm}\label{thm: oh simeq ot}
Let $\mu_k$ be such that $\mu_k^{-2}=\sigma_k^{(k)}$. Then as algebras, 
\[
\oh[\mu_1,...,\mu_N]\simeq\ot\simeq\ugl.\]
\end{thm}

\subsection{Obtaining $\oh$ from $\ogl$ via braiding}

$\ogl$ is \textit{co-quasitriangular}, which means that its category of comodules forms a braided tensor category. Following \cite{KlimykSchmudgen}[Chapter 10], this coquasitriangular structure comes from the following \textit{skew bicharacter}:
\begin{equation}
	\br: \ogl\otimes\ogl\rightarrow\mc, \qquad (\id\otimes\id\otimes\br)(X_{13}X_{24})=\Sigma\circ \hr,
\end{equation}	
where $\Sigma$ is the flip map $v\otimes w\mapsto w\otimes v$, and 
\[
\Sigma\circ\hr = \sum\limits_{i,j}q^{-\delta_{ij}}e_{ii}\otimes e_{jj}+(q^{-1}-q)\sum\limits_{i<j}e_{ij}\otimes e_{ji}.\]	
The skew bicharacter satisfies
\begin{equation}
	\br(ab,c) = \br(a,c_1)\br(b,c_2), \quad \br(a,bc)=\br(a_2,b)\br(a_1,c), \quad \br(a,1)=\varepsilon(a), \quad \br(1,a)=\varepsilon(a).
\end{equation}
The skew bicharacter has a convolution inverse $\br^{-1}$ given by
\[
(\id\otimes\id\otimes\br^{-1})X_{13}X_{24} = \hat{R}^{-1}\circ\Sigma = \sum\limits_{i,j}q^{\delta_{ij}}e_{ii}\otimes e_{jj}+(q-q^{-1})\sum\limits_{i<j}e_{ij}\otimes e_{ji},\] 
which satisfies
\begin{gather}
	\br(a_1,b_1)\br^{-1}(a_2,b_2) = \br^{-1}(a_1,b_1)\br(a_2,b_2) = \varepsilon(a)\varepsilon(b) \\
	\br^{-1}(a,b) = \br(S(a),b) \\
	\br^{-1}(ab,c) = \br^{-1}(b,c_1)\br^{-1}(a,c_2), \qquad \br^{-1}(a,bc) = \br^{-1}(a_1,b)\br^{-1}(a_2,c)
\end{gather}
We also need to consider the following convolution inverse (with respect to $(\ogl,\Delta)\otimes(\ogl,\Delta^{op})$):
\begin{equation}
	(\id\otimes\id\otimes \br')(X_{13}X_{24}) = \sum\limits_{i,j}q^{\delta_{ij}}e_{ii}\otimes e_{jj}+(q-q^{-1})\sum\limits_{i<j}q^{2(j-i)}e_{ij}\otimes e_{ji}.
\end{equation}
It satisfies
\begin{gather}
	\br'(a_1,b_2)\br(a_2,b_1) =\br(a_1,b_2)\br'(a_2,b_1) = \varepsilon(a)\varepsilon(b), \\ 
	\br'(a,b) = \br(a,S(b)) \\	
	\br'(ab,c) = \br'(a,c_2)\br'(b,c_1), \quad \br'(a,bc)=\br'(a_1,b)\br'(a_2,c), \quad \br'(a,1)=\varepsilon(a), \quad \br'(1,a)=\varepsilon(a).
\end{gather}
In terms of the generating matrix, we have
\[
X_{13}X_{24} = \sum\limits_{a,b,c,d}e_{ab}\otimes e_{cd}\otimes x_{ab}x_{cd},\]
so the bicharacters are given by:
\begin{equation}
	\br(x_{ab},x_{cd}) = \begin{cases}
		q^{-\delta_{ac}} & a=b, ~ c=d \\
		\delta_{a<b}(q^{-1}-q) & b=c, ~ a=d\\
		0 & \text{otherwise}
	\end{cases}
\end{equation}
\begin{equation}
	\br^{-1}(x_{ab},x_{cd}) = \begin{cases}
		q^{\delta_{ac}} & a=b, ~ c=d \\
		\delta_{a<b}(q-q^{-1}) & b=c, ~ a=d\\
		0 & \text{otherwise}
	\end{cases}
\end{equation}
\begin{equation}
	\br'(x_{ab},x_{cd}) = \begin{cases}
		q^{\delta_{ac}} & a=b, ~ c=d \\
		\delta_{a<b}(q-q^{-1})q^{2(b-a)} & b=c, ~ a=d\\
		0 & \text{otherwise}.
	\end{cases}
\end{equation}
For the following, see \cite{Majid}[Section 7.4], or \cite{KlimykSchmudgen}[Section 10.3]:
\begin{thm}\label{thm: REA braided isomorphism}
	We can define a new product $\ast$ on $\omn$ via
	\begin{equation}\label{eq: braided product}
		f\ast g :=\br(f_1,g_2)f_2g_3\br'(f_3,g_1).
	\end{equation}
	Then with respect to this new product, there is an isomorphism of algebras
	\begin{equation}
		\Phi:\oh\rightarrow(\omn,\ast), \qquad \Phi:Z\mapsto X.
	\end{equation}
\end{thm}
We can also write the original product on $\omn$ in terms of the modified product as follows:
\begin{equation}\label{eq:original product}
	fg = \br^{-1}(f_1,g_1)f_2\ast g_3\br(f_3,g_2).
\end{equation}
\begin{defin}
	We denote by $\phi_{\ast}$ the vector space isomorphism
	\[
	\phi_{\ast}:\omn\rightarrow(\omn, \ast).\]
\end{defin}
Then combining with Theorem $\ref{thm: REA braided isomorphism}$, we have a vector space isomorphism
\[
\Phi^{-1}\circ\phi_{\ast}:\omn\rightarrow\oh.\]

\subsection{Quantum Minors and Coactions}

Using the relation between $\omn$ and $\oh$, we can take the quantum minors $X_{IJ}$ in $\omn$ and obtain corresponding quantum minors in $\oh$:
\begin{defin}
For $I,J\subseteq[N]$, $\lvert I\rvert=\lvert J\rvert$, the \textbf{quantum minor} $Z_{IJ}\in\oh$ is defined by
\[
Z_{IJ}:=\Phi^{-1}\left(\phi_{\ast}(X_{IJ})\right).\]
\end{defin}
Explicit expansion formulae for the $\oh$ quantum minors are given in \cite{DCM1}[Proposition 2.9]. A particularly important property of $\oh$ is that it forms a comodule algebra over both $\ogl$ and $\ugl$ \cite{DCM1}. For our purposes, we will only consider the $\ugl$ case. The comodule algebra structure means that there is a homomorphism 
\[
\beta:\oh\rightarrow\oh\otimes\ugl^{\text{cop}},\]
such that $\oh$ forms a comodule over $\ugl^{\text{cop}}$ with respect to $\beta$. This in turn means that $\Rep\oh$ forms a module category over $\Rep\ugl^{\text{cop}}$. Using the isomorphism in Theorem \ref{thm: ugl simeq ot}, combined with \cite{KlimykSchmudgen}[Proposition 10.30] or \cite{Majid}[Theorem 7.4.1], we can describe the coaction as follows:
\begin{prop}\label{prop: oh coaction to ot}
The comodule algebra structure $\beta:\oh\rightarrow\oh\otimes\ot$ is given by
\[
\beta:Z\mapsto T_{13}^lZ_{12}T_{13}^u.\]
In terms of generators, we have
\[
\beta(z_{ij}) = \sum\limits_{\substack{1\leq k\leq i\\ 1\leq l\leq j}}z_{kl}\otimes t_{ik}t_{lj}.\]
In particular, the leading quantum minors satisfy
\[
\beta(Z_{[k],[k]}) = Z_{[k],[k]}\otimes\left(\prod\limits_{i=1}^{k}t_{ii}^{2}\right), \quad 1\leq k\leq N.\]
\end{prop}

If we combine the coaction with the character $z_{ij}\mapsto\delta_{ij}$ on $\oh$, then this gives us an injective map 
\begin{gather}
\oh\rightarrow\ot \\ Z\mapsto T^lT^u\\
z_{ij}\mapsto\sum\limits_{1\leq k\leq \min\{i,j\}}t_{ik}t_{kj} \qquad Z_{[k],[k]}\mapsto \prod\limits_{i=1}^{k}t_{ii}^{2}.
\end{gather}
The isomorphism in Theorem \ref{thm: oh simeq ot} can then be seen to be the result of this injective map, and the $\mu_k$ in the theorem are given by $\mu_k=Z_{[k],[k]}^{-\frac{1}{2}}$.

\section{A Presentation of Small $\oh$}\label{section: 3}

The braided product $\ast$ can also be considered with respect to the quotient algebra $\bogl$. To do so, we take the ideal defining $\bogl$ as a quotient, and combine this with the braided product and isomorphism in Theorem \ref{thm: REA braided isomorphism}. This then gives a corresponding ideal in $\oh$, which defines a \textit{small} version of the reflection equation algebra.

\begin{defin}
The \textbf{small reflection equation algebra} $\boh$ is the finite dimensional quotient of $\oh$, with quotient relations
\[
\Phi^{-1}(\phi_{\ast}(x_{ii}^{p}))=1, \qquad \Phi^{-1}(\phi_{\ast}(x_{ij}^{\tp}))=0 ~ \text{ for } ~ i\neq j.\]
\end{defin}
A presentation for $\boh$ at odd roots of unity was given in \cite{CookeLaugwitz}:
\begin{thm}
For $q$ an odd root of unity, the quotient $\boh$ is defined by
\begin{equation}
z_{11}^{p}=1, \qquad z_{ij}^{\tp}=0 ~ \text{ for } ~ i\neq j, \qquad \Phi^{-1}(\phi_{\ast}(x_{kk}^{p}))=1 ~ \text{ for } ~ 1<k\leq N.
\end{equation}
\end{thm}
Although we omit it here, in \cite{CookeLaugwitz}[Theorem 4.15] an explicit formula for $\Phi^{-1}(\phi_{\ast}(x_{kk}^{p}))$, $1<k\leq N$, is given. This formula is very complicated compared to the other relations, so our aim is to give an alternative relation equivalent to this using quantum minors. We will do this by combining Theorem \ref{thm: REA braided isomorphism} with Proposition \ref{prop: power of quantum determinant, minor expansion}.

For this, we will first need several lemmas. To simplify notation, we use the following:
\begin{defin}
	For $a\in\omn$, and $\ast$ the braided product from Equation \ref{eq: braided product}, we denote
	\[
	a^{\ast n}:=a\ast a\ast...\ast a.\]
\end{defin}
The following result appears in \cite{CookeLaugwitz}[Proposition 4.23]: 
\begin{lemma}\label{lem: ast expansion 2}
	For $i\neq j$, and $n\in\mn$, we have
	\[
	x_{ij}^{n} = q^{\frac{1}{2}(n^2-n)}x_{ij}^{\ast n}.\]
\end{lemma}
Next, we will need a similar result for more general quantum minors:
\begin{lemma}\label{lem: ast expansion 1}
Let $I=\{1,2,...,k\}:=[k]$, and $J=\{j_1,...,j_k\}\subset\{1,...,N\}=[N]$. Then
\[
X_{IJ}^{n}=q^{\frac{1}{2}(n^2-n)\left(k-\lvert I\cap J\rvert\right)}X_{IJ}^{\ast n}.\]
\end{lemma}
\begin{proof}
Starting with $n=2$, and using the product formula in Equation \ref{eq:original product}, we have
\[
X_{IJ}^2 = \sum\limits_{A,B,C,D}\br^{-1}(X_{IA},X_{IC})X_{AB}\ast X_{DJ}\br(X_{BJ},X_{CD}),\]
where the sum is over all subsets $A,B,C,D\subset\{1,...,N\}$ of size $k$. In what follows, we switch to using Sweedler notation. By \cite{DCM1}[Lemma 2.7], we see that the only non-zero terms will occur when $A=C=I$, which in turn requires $D=I$ and $B=J$.
Hence it simplifies to
\[
X_{IJ}^{2} = \br^{-1}(X_{II},X_{II})X_{IJ}\ast X_{IJ}\br(X_{JJ},X_{II}).\]
We further have
\[
\br^{-1}(X_{II},X_{II}) = q^{k}, \qquad \br(X_{JJ},X_{II}) = q^{-\lvert I\cap J\rvert},\]
and so get
\[
X_{IJ}^{2} = q^{k-\lvert I\cap J\rvert} X_{IJ}^{\ast 2}.\]
Next, for general $n$, we have
\[
X_{IJ}^{n} = \br^{-1}(X_{IA},X_{IC_1}...X_{IC_{n-1}})X_{AB}\ast(X_{D_1J}...X_{D_{n-1}J})\br(X_{BJ},X_{C_1D_1}...X_{C_{n-1}D_{n-1}})\]
\begin{gather*}
=\br^{-1}(X_{IA_1},X_{IC_1})\br^{-1}(X_{A_1A_2},X_{IC_2})...\br^{-1}(X_{A_{n-2}A_{n-1}},X_{IC_{n-1}})X_{A_{n-1}B}\ast(X_{D_1J}...X_{D_{n-1}J}) \times \\ \times \br(X_{B_1B_2},X_{C_{n-1}D_{n-1}})...\br(X_{B_{n-2}B_{n-1}},X_{C_2D_2})\br(X_{B_{n-1}J},X_{C_1D_1}).
\end{gather*}
Again by \cite{DCM1}[Lemma 2.7], we must have
\[
I = C_1=...=C_{n-1} = A_1=...=A_{n-1},\]
which in turn forces 
\[
I = D_1 =...= D_{n-1}, \qquad J=B_1=...=B_{n-1}.\]
The above relation therefore simplifies to 
\[
X_{IJ}^n = q^{(n-1)(k-\lvert I\cap J\vert)}X_{IJ}\ast(X_{IJ}^{n-1}),\]
and hence by induction we get
\[
X_{IJ}^{n} = q^{\frac{1}{2}(n^2-n)\left(k-\lvert I\cap J\rvert\right)}X_{IJ}^{\ast n}.\]
\end{proof}
In particular, we note that for $I=J=[k]$, the lemma gives
\[
X_{[k],[k]}^{n} = X_{[k],[k]}^{\ast n}.\]

\begin{lemma}\label{lem: ast expansion 3}
Let $1\leq i,j\leq N$ and $n\in\mn$. Then	
\[
x_{Ni}^n X_{[N-1],[N]\setminus\{j\}}^n = q^{n^2(\delta_{iN}-1)}x_{Ni}^n \ast X_{[N-1],[N]\setminus\{j\}}^n.\]
\end{lemma}
\begin{proof}
By associativity of the respective products, we can start the expansion in the middle. We have
\begin{gather*}
x_{Ni}^n X_{[N-1],[N]\setminus\{j\}}^n = \sum\limits_{a_l,b_l,C_l,D_l}\br^{-1}(x_{Na_1}...x_{Na_{n}},X_{[N-1],C_1}...X_{[N-1],C_{n}}) \times \\ \times (x_{a_1b_1}...x_{a_{n}b_{n}})\ast(X_{D_1,[N]\setminus\{j\}.}...X_{D_{n},[N]\setminus\{j\}})\br(x_{b_1i}...x_{b_{n}i},X_{C_1D_1}...X_{C_{n}D_{n}}).		
\end{gather*}
Expanding this out, we get
\begin{gather*}
= \sum\br^{-1}(x_{Na_{1,1}}...x_{Na_{n,1}},X_{[N-1],C_1})...\br^{-1}(x_{a_{1,n-1}a_{1,n}}...x_{a_{n,n-1}a_{n,n}},X_{[N-1],C_{n}})(x_{a_1b_1}...x_{a_{n}b_{n}})\ast \\ \ast (X_{D_1,[N]\setminus\{j\}.}...X_{D_{n-1},[N]\setminus\{j\}})\br(x_{b_{1,1}b_{1,2}}...x_{b_{n,1}b_{n,2}},X_{C_{n}D_{n}})...\br(x_{b_{1,n}i}...x_{b_{n,n}i},X_{C_1D_1}),
\end{gather*}
where $a_{i,n}=a_i$, $b_{i,1}=b_i$. Expanding further gives
\begin{gather*}
= \sum\br^{-1}(x_{Na_{1,1}},X_{C_{1,n-1},C_{1,n}})...\br^{-1}(x_{Na_{n,1}},X_{[N-1],C_{1,1}})...\br^{-1}(x_{Na_{n,n}},X_{[N-1],C_{n,1}})\times \\ \times
(x_{a_1b_1}...x_{a_{n}b_{n}})\ast (X_{D_1,[N]\setminus\{j\}.}...X_{D_{n},[N]\setminus\{j\}})\times \\ \times
\br(x_{b_{1,1}b_{1,2}},X_{C_{n}D_{n,1}})...\br(x_{b_{n,1}b_{n,2}},X_{D_{n,n-1}D_{n,n}})...\br(x_{b_{n,n}i},X_{D_{1,n-1}D_{1,n}}),
\end{gather*}
where $C_{i,n}=C_i$, $D_{i,n}=D_i$. Again by \cite{DCM1}[Lemma 2.7], we see that the only non-zero terms occur when 
\[
a_{l,m}=N, \quad b_{l,m}=i, \quad C_{l,m}=D_{l,m}=[N-1]. \]
Hence it reduces to
\[
x_{Ni}^nX_{[N-1],[N]\setminus\{j\}} = \left(\br^{-1}(x_{NN},X_{[N-1],[N-1]})\br(x_{ii},X_{[N-1],[N-1]})\right)^{n^2}(x_{Ni}^n)\ast(X_{[N-1],[N]\setminus\{j\}}^n),\]
which then gives the required identity.
\end{proof}
For $1\leq i<N$ and $1\leq j\leq N$, we can combine Lemmas \ref{lem: ast expansion 2}. \ref{lem: ast expansion 1}, and \ref{lem: ast expansion 3} to get the following identity:
\begin{equation}
x_{Ni}^{n}X_{[N-1],[N]\setminus\{j\}}^n = q^{\frac{1}{2}\delta_{jN}(n-n^2)-n}x_{Ni}^{\ast n}\ast X_{[N-1],[N]\setminus\{j\}}^{\ast n}.
\end{equation}
The final lemma we need is the following:
\begin{lemma}\label{lem: ast expansion 4}
When $q$ is a root of unity, 
\[
x_{NN}^{2\tp} = (x_{NN}^{\tp})\ast(x_{NN}^{\tp})+\sum\limits_{i,j=1}^{N-1}c_{i,j}(x_{Ni}^{\tp})\ast(x_{jN}^{\tp}),\]
for some coefficients $c_{i,j}$.
\end{lemma}
\begin{proof}
We first consider $x_{NN}^{2n}$ for a general choice of $n\in\mn$. Applying the braided product formula in the middle, we have
\[
x_{NN}^{2n} = \sum\br^{-1}(x_{Na_1}...x_{Na_n},x_{Nc_1}...x_{Nc_n})(x_{a_1b_1}...x_{a_nb_n})\ast(x_{d_1N}...x_{d_nN})\br(x_{b_1N}...x_{b_nN},x_{c_1d_1}...x_{c_nd_n}).\]
Expanding the $\br^{-1}$ term first, we get
\begin{gather*}
=\sum\br^{-1}(x_{Na_{1,1}},x_{c_{1,n-1}c_{1,n}})...\br^{-1}(x_{Na_{n,1}},x_{Nc_{1,1}})...\br^{-1}(x_{Na_{n,n}},x_{Nc_{n,1}})\times \\ \times
(x_{a_1b_1}...x_{a_nb_n})\ast(x_{d_1N}...x_{d_nN})\br(x_{b_1N}...x_{b_nN},x_{c_1d_1}...x_{c_nd_n}),
\end{gather*}
where $a_{i,n}=a_i$, $c_{i,1}=c_i$. We see the only non-zero terms occur when $a_{l,m}=c_{l,m}=N$, and so it becomes
\begin{gather}\label{eq: lemma partial split}
x_{NN}^{2n} = \sum\limits_{\substack{b_1,...,b_n=1\\ d_1,...,d_n=1}}^{N}q^{n^2}(x_{Nb_1}...x_{Nb_n})\ast(x_{d_1N}...x_{d_nN})\br(x_{b_1N}...x_{b_nN},x_{Nd_1}...x_{Nd_n}).
\end{gather}	
To motivate how we will simplify this, we first consider a simple case with $n=N=2$. Then
\[
q^{-4}X_{22}^{4} = \sum\limits_{1\leq b_1,b_2,d_1,d_2\leq 2}(x_{2b_1}x_{2b_2})\ast(x_{d_12}x_{d_22})\br(x_{b_12}x_{b_22},x_{2d_1}x_{2d_2}).\]
There will be four terms where $\{b_1,b_2\}=\{d_1,d_2\}=\{1,2\}$. They are
\begin{gather*}
(x_{22}x_{21})\ast(x_{22}x_{12})\br(x_{22}x_{12},x_{22}x_{21}), \quad (x_{21}x_{22})\ast(x_{22}x_{12})\br(x_{12}x_{22},x_{22}x_{21}), \\ (x_{22}x_{21})\ast(x_{12}x_{22})\br(x_{22}x_{12},x_{21}x_{22}), \quad (x_{21}x_{22})\ast(x_{12}x_{22})\br(x_{12}x_{22},x_{21}x_{22}).
\end{gather*} 
Using the relations
\[
x_{ij}x_{ik}=qx_{ik}x_{ij}, \quad x_{ji}x_{ki}=qx_{ki}x_{ji}, ~ \text{ for } j<k,\]
We can simplify those four terms to get
\[
(1+q^2)^2(x_{22}x_{21})\ast(x_{22}x_{12})\br(x_{22}x_{12},x_{22}x_{21}).\]
We then see that for $q=\pm i$, the coefficient equals zero.
We now want to consider this more generally. Starting from equation \ref{eq: lemma partial split}, consider a fixed choice of positive integers satisfying
\[
N\geq b_1\geq b_2\geq...\geq b_n\geq 1, \quad N\geq d_1\geq d_2\geq ...\geq d_2\geq 1.\]
Then Equation \ref{eq: lemma partial split} will contain a sub-sum over all permutations of these integers. By simplifying using the above $\omn$ relations, we will obtain a term of the form
\[
C_{b_1,...b_n}C_{d_1,...,d_n}(x_{Nb_1}...x_{Nb_n})\ast(x_{d_1N}...x_{d_nN})\br(x_{b_1N}...x_{b_nN},x_{Nd_1}...x_{Nd_n}).\]
We call the coefficients $C_{b_1,...,b_n}$, $C_{d_1,...,d_n}$, \textit{reordering coefficients}. Our aim is to show that in the case $n=\tp$, the majority of reordering coefficients will evaluate to zero.

To calculate the reordering coefficients, consider the algebra $Y_N$ generated by elements $y_1,...,y_N$ with relation
\[
y_iy_j = q^2 y_jy_i ~ \text{ for }i<j.\]
Then it can be seen that the reordering coefficient $C_{b_1,...,b_n}$ is given by the corresponding coefficient of the term $y_{b_1}...y_{b_n}$ in the expansion of 
\[
(y_1+...+y_N)^n.\]
To describe this expansion, write $y_{1,k}:=\sum\limits_{i=1}^{k}y_i$. Then following \cite{ParshallWang}[Chapter 7],
the expansion satisfies
\[
(y_1+...+y_N)^n = (y_{1,N-1}+y_N)^n = \sum\limits_{m=0}^{n}\binom{n}{m}_{q^2}y_{N}^{m}y_{1,N-1}^{n-m},\]
where $\binom{n}{m}_{q^2}$ is a quantum binomial coefficient. We now specialize to the case $n=\tp$. Then by \cite{ParshallWang}[Corollary 7.1.3], we have
\[
\binom{\tp}{m}_{q^2}=0 \quad \text{ for } 0<m<\tp, \qquad \binom{\tp}{0}_{q^2}=\binom{\tp}{\tp}_{q^{2}}=1.\]
Hence it simplifies to 
\[
(y_1+...+y_N)^{\tp} = y_{1,N-1}^{\tp}+y_{N}^{\tp}.\]
By induction on $y_{1,k}$, we then get
\[
(y_1+...+y_N)^{\tp}=\sum\limits_{i=1}^{N}y_i^{\tp}.\]
We see from this that the only non-zero reordering coefficients in Equation \ref{eq: lemma partial split}  will be $C_{i,i,...,i}$, for $1\leq i\leq N$. Hence the equation simplifies to give
\[
x_{NN}^{2\tp} = \sum\limits_{i,j=1}^{N}q^{\tp^2}(x_{Ni}^{\tp})\ast(x_{jN}^{\tp})\br(x_{iN}^{\tp},x_{Nj}^{\tp}),\]
\[
= (x_{NN}^{\tp})\ast(x_{NN}^{\tp})+\sum\limits_{i,j=1}^{N-1}q^{\tp^2}(x_{Ni}^{\tp})\ast(x_{jN}^{\tp})\br(x_{iN}^{\tp},x_{Nj}^{\tp}).\]
\end{proof}
We can now proceed to giving our presentation of $\boh$:
\begin{thm}\label{thm: alternative presentation}
The small reflection equation algebra $\boh$ is given as a quotient of $\oh$ by the relations
\[
Z_{[k],[k]}^p=1, ~ 1\leq k\leq N, \qquad z_{ij}^{\tp}=0, ~ i\neq j.\]
\end{thm}
\begin{proof}
We first note that the relations defining $\boh$ are given by taking the corresponding relations defining $\bogl$ and applying $\Phi^{-1}\circ\phi_{\ast}$ to them. In particular, $\phi_{\ast}$ only depends on $\br$ and $\br^{-1}$, and so is valid when specialized to any root of unity.

The relations
\begin{gather}\label{eq: partial boh rels}
Z_{[1],[1]}^p=z_{11}^p=1, \qquad z_{ij}^{\tp}=0 ~ \text{ for } ~ i\neq j
\end{gather} 
are the same as those given in \cite{CookeLaugwitz}, and follow similarly for even roots of unity via Lemmas \ref{lem: ast expansion 2} and \ref{lem: ast expansion 1}. We also note that combining Proposition \ref{prop:power of small quantum determinant}, along with the fact that $Z_{IJ}=\Phi^{-1}(\phi_{\ast}(X_{IJ})))$ by definition, we see that $Z_{[k],[k]}$ is invertible in $\boh$ for all $k\in\mn$. It follows that all that remains now is to show that the relations $Z_{[k],[k]}^{p}=1$ are equivalent to those given by $\Phi^{-1}(\phi_{\ast}(x_{kk}^{p}))=1$.\\

We first note that in $\bogl$, by Proposition \ref{prop:power of small quantum determinant}, the relations $x_{kk}^{p}=1$, $1\leq k\leq N$, are in fact equivalent to the relations $X_{[k],[k]}^p=1$. Hence we want to obtain a version of this relation in $\boh$. To do so, we use the identity in Proposition \ref{prop: power of quantum determinant, minor expansion}, and apply the braided product to it. The identity in the proposition is
\begin{gather*}
D_q^{\tp} = \sum\limits_{i=1}^{N}(\pm 1)^{i-1}x_{Ni}^{\tp}X_{[N-1],[N]\setminus\{i\}}^{\tp}.
\end{gather*}
We then want to consider the corresponding relation in $\oh$, i.e. compute $\Phi^{-1}(\phi_{\ast}(D_q^{\tp}))$. Using Lemmas \ref{lem: ast expansion 2}, \ref{lem: ast expansion 1}, and \ref{lem: ast expansion 3}, gives
\[
D_q^{\ast \tp} = \phi_{\ast}(x_{NN}^{\tp})\ast X_{[N-1],[N-1]}^{\ast\tp}+\sum\limits_{i=1}^{N-1}\pm x_{Ni}^{\ast \tp}\ast X_{[N-1],[N]\setminus\{i\}}^{\ast \tp}.\]
Further applying the isomorphism $\Phi^{-1}$ from Theorem $\ref{thm: REA braided isomorphism}$, we get
\[
Z_{[N],[N]}^{\tp} = \Phi^{-1}(\phi_{\ast}(x_{NN}^{\tp}))Z_{[N-1],[N-1]}^{\tp}+\sum\limits_{i=1}^{N-1}\pm z_{Ni}^{\tp}Z_{[N-1],[N]\setminus\{i\}}^{\tp}.
\]
Using the relations in \ref{eq: partial boh rels}, along with invertibility of $Z_{[k],[k]}$, we can then simplify this, giving
\begin{gather}\label{eq: tp znn relation} 
\Phi^{-1}(\phi_{\ast}(x_{NN}^{\tp})) = Z_{[N],[N]}^{\tp}Z_{[N-1],[N-1]}^{-\tp}.
\end{gather}
For $q$ an odd root of unity, this then gives
\[
Z_{[N],[N]}^{p}Z_{[N-1],[N-1]}^{-p}=1,\]
which simplifies to the required relation. For $q$ an even root of unity, we further need to calculate $\Phi^{-1}(\phi_{\ast}(x_{NN}^{p}))$. Using Lemma \ref{lem: ast expansion 4}, along with relation \ref{eq: tp znn relation}, we get
\[
\Phi^{-1}(\phi_{\ast}(x_{NN}^{p})) = Z_{[N],[N]}^{\tp}Z_{[N-1],[N-1]}^{-\tp}Z_{[N],[N]}^{\tp}Z_{[N-1],[N-1]}^{-\tp}+\sum\limits_{i,j=1}^{N-1}q^{\tp(\tp-1)}c_{i,j}z_{Ni}^{\tp}z_{jN}^{\tp},\]
which then gives
\[
\Phi^{-1}(\phi_{\ast}(x_{NN}^{p})) = Z_{[N],[N]}^{p}Z_{[N-1],[N-1]}^{-p}=1.\]
\end{proof}
\begin{exmp}
As an example, for $q$ an odd root of unity with $p=3$, we have
\begin{gather*}
Z_{12,12}^3=(z_{22}z_{11}-q^{-2}z_{21}z_{12})^3 \\
= z_{22}^3z_{11}^3-(q^{-2}+1+q^2)z_{22}^2z_{21}z_{12}z_{11}^2+(1+q^{-2}+q^{-4})z_{22}z_{21}^2z_{12}^2z_{11}+(2q^6-q^4-q^2)z_{22}z_{21}z_{12}z_{11}^3\\
+(1-q^6)z_{21}^2z_{12}^2z_{11}^2-q^6(q^2-1)^2z_{21}z_{12}z_{11}^4-q^{-6}z_{21}^3z_{12}^3. 
\end{gather*}
Via the other $\boh$ relations, this simplifies to give the relation
\begin{gather*}
Z_{12,12}^3z_{11}^{-3}= z_{22}^3+(2-q^4-q^2)z_{22}z_{21}z_{12}-(q^2-1)^2z_{21}z_{12}z_{11}=1.
\end{gather*}
Then this can be seen to match the relation given in \cite{CookeLaugwitz}[Example 4.20].
\end{exmp}
From our presentation, we can also see (after mapping $q\mapsto q^{-1}$) that the algebra $\mathcal{L}_{0,1}(\bar{U}_q)$ in \cite{Faitg}[Theorem 4.4.1] is isomorphic to $\bohtwo$ with the added condition $Z_{12,12}=1$.

As an immediate consequence of our alternative presentation in Theorem \ref{thm: alternative presentation}, we get the following:
\begin{corr}\label{corr: boh simeq bgl}
For odd roots of unity, the small reflection equation algebra $\boh$ is isomorphic to $\bgl$. For even roots of unity, if we extend $\boh$ by the elements $Z_{[k],[k]}^{\frac{1}{2}}$, $1\leq k\leq N$, then $\boh[Z_{[1],[1]}^{\frac{1}{2}},...,Z_{[N],[N]}^{\frac{1}{2}}]$ is isomorphic to $\bglt$.
\end{corr}
\begin{proof}
We first note that $\phi_{\ast}$ preserves dimensions, so 
\[
\dim \bogl = \dim \boh = p^N\tp^{N^2-N}.\]
Then the result follows immediately from the presentation in Theorem \ref{thm: alternative presentation}, combined with the isomorphism in Theorem \ref{thm: oh simeq ot}. We note that in the even root of unity case, we have $\bglt$ appearing instead of $\bgl$, as if $Z_{[k],[k]}^p=1$, then the square roots of the quantum minors will have order $2p$, and as the isomorphism maps 
\[
Z_{[k],[k]}^{\frac{1}{2}}\mapsto\prod\limits_{i=1}^{k}K_{i}^{-1},\] 
we require $K_i^{2p}=1$ to match this.

In the odd root of unity case, we can see that the addition of square roots of quantum minors is not necessary, as in this case we will have
\[
Z_{[k],[k]}^{\frac{p+1}{2}}\mapsto \prod\limits_{i=1}^{k}K_i^{-1}.\] 
\end{proof}

\subsection{Representations of the Small Reflection Equation Algebra}

Theorem \ref{thm: alternative presentation} allows for a straightforward classification of $\boh$ irreducibles. We note that the case of $N=2$ follows from \cite{BeraMukherjee}, and via the isomorphism with $\bgl$, the general case can also be related to \cite{Lusztig}[Proposition 5.11]. Before giving the classification, we want to use the general properties of $\oh$ quantum minors to describe how $\boh$ irreducibles can be viewed as highest weight representations. We will need the following relations in $\oh$:
\begin{prop}
Choose $I,J\subseteq[N]$, $\lvert I\rvert=\lvert J\rvert$, and $1\leq k\leq N$. We have the following identities in $\oh$:
\begin{gather}
	\quad Z_{[k],[k]}Z_{IJ} = q^{2\lvert I\cap [k]\rvert-2\lvert J\cap [k]\rvert}Z_{IJ}Z_{[k],[k]}, \label{eqn: Zkk q commute} \\
	Z_{[k],[k+1]\setminus\{k\}}Z_{[k+1]\setminus\{k\},[k]}-q^{-2}Z_{[k+1]\setminus\{k\},[k]}Z_{[k],[k+1]\setminus\{k\}} = (1-q^2)(Z_{[k-1],[k-1]}Z_{[k+1],[k+1]}-Z_{[k],[k]}^2) \label{eqn: zij relation} \\
	z_{1k}Z_{[k],[k+1]\setminus\{k\}}-q^{-1}Z_{[k],[k+1]\setminus\{k\}}z_{1k}=(q^2-1)z_{1,k+1}Z_{[k],[k]} \label{eqn: zij relation 2}\\
	z_{k1}z_{1,k+i}-qz_{1,k+i}z_{k1} = (q^2-1)z_{k,k+i}z_{11}, \quad i\geq 1. \label{eqn: zij relation 3} 
\end{gather}
\end{prop}
\begin{proof}
The first equation is given in \cite{DCM1}[Equation 2.21]. The second equation follows from \cite{Me1}[Lemma 4.4]. For the third relation, we use \cite{DCM1}[Equation 2.20] with $I=[k]$, $J=[k+1]\setminus\{k\}$, $I'=k$, $J'=1$. The last equation comes directly from the $\oh$ relations. 
\end{proof}
Let $V$ be an $\boh$ irreducible. Then Equation \ref{eqn: Zkk q commute} says that the $Z_{[k],[k]}$ $1\leq k\leq N$ can be simultaneously diagonalized on $V$, and that every other $z_{ij}$, $i\neq j$, or $Z_{IJ}$, $I\neq J$, will act as a lowering or raising operator on $V$. Hence $V$ can be viewed as a highest weight representation. By combining Equations \ref{eqn: zij relation}, \ref{eqn: zij relation 2}, and \ref{eqn: zij relation 3}, we see that the $\boh$ action on $V$ is fully determined by the possible highest weights. From our alternative presentation of $\boh$ in Theorem \ref{thm: alternative presentation}, we see that if $Z_{[k],[k]}$ acts on a highest weight vector $v_0$ of $V$ by
\[
Z_{[k],[k]}v_0 = \lambda_kv_0, \quad \text{ for  } \lambda_k\in\mc, \quad 1\leq k\leq N, \]
then $\lambda_k$ must satisfy
\[
\lambda_k^p=1.\] 
It follows from this that the possible highest weights will be in bijection with $\mz_p^N$. Hence all that remains is to verify that an irreducible $\boh$ representation actually exists for each element of $\mz_p^N$. We can verify this and state the classification of irreducibles as follows:
\begin{thm}\label{thm: boh rep classification}
The irreducible representations of $\boh$ are in bijection with the set $\mz_p^N$.
\end{thm}
\begin{proof}
Choose $\lambda=(\lambda_1,...,\lambda_N)\in\mc^N$ such that $\lambda_k^p=1$ for $1\leq k\leq N$. To verify the existence of an $\boh$ irreducible $V_{\lambda}$ corresponding to $\lambda$, we note that $V_{\lambda}$ can be considered as the irreducible quotient of the specialization of a \textit{big cell representation} from \cite{DCM1}[Theorem 4.15] to $q$ a root of unity, where we only need to consider the $\epsilon=(1,...,1)$ case. Given such a choice of $\lambda$, we can always choose $r\in\mr^N$ such that we have
\[
\lambda_k=q^{2\sum\limits_{i=1}^{k}r_i}.\]
To relate this to the ordered highest weights in \cite{DCM1}[Theorem 4.15], we see that we can always choose $k_i\in\mn$ such that 
\[
r':=(r_1',...,r_N')=(r_1+k_1p,...,r_{N-1}+k_{N-1}p,r_N)\]
satisfies 
\[
r_1'\geq r_2'\geq...\geq r_N'.\]
Then $q^{r_i'}=q^{r_i}$ for each $i$, and $Z_{[k],[k]}$ acts on a highest weight vector $v_0$ as
\[
Z_{[k],[k]}v_0 = q^{2\sum\limits_{i=1}^{k}r_i'}v_0 = q^{2\sum\limits_{i=1}^{k}r_i}v_0 = \lambda_k v_0.\] 
Hence we see that the $\boh$ irreducible corresponding to $r$ is a quotient of the specialization (to $q$ a root of unity) of the $\ot$ big cell representation of highest weight $r'$.
\end{proof}

\section{Wee Reflection Equation Algebras}\label{section: 4}

Recall we have the coaction
\[
\beta:\oh\rightarrow\oh\otimes\ugl^{\text{cop}}.\]
This coaction then descends to the respective quotients, giving a coaction
\[
\beta:\boh\rightarrow\boh\otimes \bgl^{\text{cop}},\]
where we omit which version of $\bglk$ we need for brevity. In \cite{DCM2}, a notion of \textit{shape} invariant of $\oh$ representations was introduced. Given an $\oh$ irreducible $W$, the shape $\mcs$ associated to $W$ can be thought of as describing an ideal $\mci_{\mcs}$ (see \cite{DCM2}[Definition 3.5]) such that $\mci_{\mcs}$ is contained in the kernel of $\oh\rightarrow\text{End}(W)$. This shape invariant has some interesting properties: Firstly, the shape is preserved by the $\ugl$ coaction \cite{Me1}[Proposition 6.4], i.e. the homomorphism $\beta$ maps
\[
\beta:\oh/\mci_{\mcs}\rightarrow\left(\oh/\mci_{\mcs}\right)\otimes\ugl.\]
This means that if $V$ is a $\ugl$ representation, then $W\otimes V$ will be an $\oh$ representation with the same shape as $W$. Secondly, for a given shape $\mcs$ there are elements $Z_{\mcs,k}\in\oh$ such that they $q$-commute with all elements in $\oh/\mci_{\mcs}$. In terms of the coaction, they satisfy
\[
\beta:Z_{\mcs,k}\mapsto Z_{\mcs,k}\otimes\left(\prod K_{i}^{-1}\right),\]
where the product is over a certain labelling associated to the shape. 

The simplest example of shape is the identity shape $\mcs=id$. Any representation with this shape, can be thought of as a $\ugl$ representation coming from the embedding $\oh\rightarrow\ugl$ in Theorem \ref{thm: oh simeq ot}. In the case of the identity shape, we simply have $Z_{\mcs,k}=Z_{[k],[k]}$. We can then think of the algebra $\boh$ as a quotient of $\oh$ that is \textit{compatible} with the identity shape ideal $\mci_{\mcs=\id}$. Motivated by this, it is then natural to ask if we can introduce alternative quotients of $\oh$ that are compatible with other shapes. We define this in a general way as follows:
\begin{defin}
Let $A$ be a finite dimensional quotient of $\oh$. We call $A$ a \textbf{wee reflection equation algebra} if $A$ forms a comodule algebra over $\bgl$, (or alternatively $\bsl$ or $\bglk$) with respect to the coaction $\beta$.
\end{defin}
In particular, if $\mci_H$ is an ideal of $\oh$, with $A=\oh/\mci_H$, and $\mci_U$ is the ideal of $\ugl$ such that $\ugl/\mci_U = \bgl$, then for $A$ to be a wee REA, we require
\[
\beta(\mci_H)\subseteq \mci_H\otimes \ugl + \oh\otimes \mci_U.\]
An immediate consequence of the definition, is that given a wee REA $A$, then $\Rep A$ will form a module category over $\Rep \bgl$. Considering our motivation of the shape invariant, we can further specify when a wee REA is associated to a particular shape:
\begin{defin}
Given a wee REA $A$ with $A=\oh/\mci_{H}$, we say it is associated to a shape $\mcs$ if $\mci_{\mcs}\subseteq\mci_H$, and $Z_{\mcs,k}\notin\mci_H$ for all $k$.
\end{defin}

In the motivating case of a wee REA associated to a shape, it is straightforward to see that we expect relations of the form $Z_{\mcs,k}^{p}-1\in\mci_H$. Presumably, we also require other quotient relations, similar to $z_{ij}^{\tp}=0$, to ensure the resulting algebra is finite dimensional. Unfortunately, in the case of general shapes, it appears (based on \cite{Me1}[Lemma 3.8]) to be quite difficult to determine what quotient relations are actually necessary, and so we leave this to future work. For now, we will instead focus on two families of examples of wee REAs. These are wee REAs of \textit{big cell type,} and wee REAs coming from \textit{exact comodule algebras} via the classification in \cite{Mombelli, NSS}.

\subsection{Wee REAs of big cell type}
In \cite{DCM1}, a family of $\oh$ representations known as \textit{big cell representations} were considered. These can be thought of as coming from a family of algebras $\ote$, which are modifications of $\ot$ based on some choice of $\epsilon\in\{1,-1,0\}^{N}$. In fact, the isomorphism in Theorem \ref{thm: oh simeq ot} generalizes to give a homomorphism
\[
\oh\rightarrow\ote,\]
and a big cell representation factors through this homomorphism. Our aim now is to similarly generalize $\boh$ based on a notion of $\bote$. By how $\epsilon$ is actually used in the definition of $\mathcal{O}_{q,\epsilon}(T(N))$ (see \cite{DCM1}[Section 4.2]), we only need to consider $\epsilon$ of the form
\[
\epsilon=(\epsilon_1,...,\epsilon_K,0,...,0), \quad \epsilon_1,...,\epsilon_K\in\{1,-1\}.\]
Hence to simplify, we will view any chosen $\epsilon$ as consisting of $K$ non-zero terms followed by $N-K$ zeroes, for some $1\leq K\leq N$. Given $\epsilon=(\epsilon_1,...,\epsilon_N)$, we denote 
\[
\epsilon_{(i,j]}:=\prod\limits_{s=i+1}^{j}\epsilon_s.\]
\begin{defin}
Given a choice of $\epsilon$, $\ote$ is the algebra generated by the sub-algebras $\otl$ and $\otu$, along with the added relation
\begin{equation}
	\begin{split}
		t_{ij}t_{kl}-q^{\delta_{jk}-\delta_{il}}t_{kl}t_{ij} =& \delta_{il}q^{\delta_{jk}-1}(1-q^2)\left(\sum\limits_{m=i+1}^{\min\{j,k\}}\epsilon_{(i,m]}t_{km}t_{mj}\right) 
		+ \delta_{jk}(q^{2}-1)\left(\sum\limits_{m=\max\{i,l\}}^{j-1}t_{im}t_{ml}\right),
	\end{split} 
\end{equation}
for $i<j$ and $k>l$.
\end{defin}
Taking $\epsilon=(1,...,1)$ then recovers the algebra $\ot$. We note that for other choices of $\epsilon$, $\ote$ is no long a Hopf algebra, but instead forms a comodule algebra over $\ot$.
The above relations generalizes, for $i<j$, $k>l$, and $n\in\mn$, to the following general commutation relation:
\begin{equation}
	\begin{split}
		t_{ij}^nt_{kl}-q^{n(\delta_{jk}-\delta_{il})}t_{kl}t_{ij}^n =& \delta_{il}q^{\delta_{jk}-1}(q^{2-2n}-q^2)t_{ij}^{n-1}\left(\sum\limits_{m=i+1}^{\min\{j,k\}}\epsilon_{(i,m]}t_{km}t_{mj}\right) \\
		&+ \delta_{jk}(q^{2n}-1)t_{ij}^{n-1}\left(\sum\limits_{m=\max\{i,l\}}^{j-1}t_{im}t_{ml}\right) \\
		&+ \delta_{il}\delta_{jk}(q^{2n}-q^2-1+q^{2-2n})t_{ij}^{n-2}\left(\sum\limits_{i+1=s\leq m}^{j-1}\epsilon_{(i,s]}t_{im}t_{ms}t_{sj}\right),
	\end{split} 
\end{equation}
where we can obtain a similar relation for $i>j$ and $k<l$ by applying the $*$-structure. We see from this relation that for $q$ a root of unity, and $n=\tp$, that in $\ote$ each of the $t_{ij}$ will again either commute or anti-commute. Hence we can define the following:
\begin{defin}
Given $k\in\mn$ and a choice of $\epsilon$, we define $\botek$ to be the finite dimensional quotient of $\ote$ by the relations
\[
t_{ii}^{kp}=1 \qquad t_{ij}^{\tp}=0, ~ i\neq j.\] 
\end{defin}
Given a choice of $\epsilon$, write $E_\epsilon$ for the $N\times N$ diagonal matrix with $\epsilon_{(0,i]}$ as the $i$th diagonal entry. Then there is a homomorphism
\begin{gather}\label{eqn: epsilon hom}
\oh\rightarrow\ote, \qquad Z\mapsto T^lE_{\epsilon}T^u.
\end{gather}
We can combine this homomorphism with the definition of $\bote$ to get a homomorphism 
\[
\oh\rightarrow\botek,\]
where $k=1$ for odd roots of unity, and $k=2$ for even roots of unity. Given this, and using \cite{DCM1}[Lemma 4.7], we then have the following:
\begin{prop}
For each choice of $1\leq K\leq N$, and $\epsilon=(\epsilon_1,...,\epsilon_K,0,,,0)$, with $\epsilon_i\in\{1,-1\}$, there is a wee reflection equation algebra given by the homomorphism $\oh\rightarrow\botek$. It is defined by the relations
\[
Z_{[k],[k]}^{p}=\prod\limits_{i=1}^{k}\epsilon^p_{(0,i]}, \quad k\leq K, \qquad \]
\[
z_{ij}^{\tp}=0, \quad i\neq j,\]
\[
Z_{IJ}=0 ~ \text{ if } ~ \lvert I\rvert >K.\]
\end{prop}
We call these quotient algebras the wee reflection equation algebras of \textit{big cell type.} By construction, they will be associated to the corresponding big cell shape. Using the same techniques as Theorem \ref{thm: boh rep classification}, we can immediately describe their irreducibles:
\begin{prop}
For $1\leq K\leq N$, and a choice of $\epsilon=(\epsilon_1,...,\epsilon_K,0,...,0)$ with $\epsilon_i\in\{1,-1\}$, the corresponding wee reflection equation algebra of big cell type has irreducibles in bijection with $\mz_p^K$.
\end{prop}

\subsection{Realizing $\bsltwo$ exact comodule algebras as wee REAs}\label{section: MNSS algebras}
The second family of wee REAs we want to consider are those related to exact comodule algebras over $\bsltwo$.
\begin{defin}
An \textbf{exact comodule algebra} $A$ over a Hopf algebra $H$ is a comodule algebra such that $\Rep A$ is exact over $\Rep H$, i.e. for any projective representation $P\in\Rep H$, and representation $W\in\Rep A$, $W\otimes P$ is projective in $\Rep A$.
\end{defin}
See \cite{AndruMombelli, EGNO} for details. 
In the case of $\bsltwo$, for $q$ an odd root of unity, the exact comodule algebras over $\bsltwo$ were classified in \cite{Mombelli}, and given explicit presentations in \cite{NSS}. Our aim here is to see which of the algebras in \cite{NSS} can be given presentations as quotients of $\ohtwo$. If the algebras can be given such a presentation, then it means that they can be viewed as a wee REA. 

To aid the reader, we try to keep notation consistent with \cite{NSS}. Let $q$ be an $p$th root of unity, for odd $p\in\mn$. Let $Div(p)$ denote the set of positive divisors of $p$. Let $r\in Div(p)$ and $\alpha,\beta,\nu,\epsilon,\zeta\in\mc$ with $(\alpha,\beta)\neq (0,0)$. For $q$ an odd root of unity, the exact comodule algebras over $\bsltwo$ were listed with explicit presentations in \cite{NSS}[Section 5.6]. The algebras are split into six families:
\begin{gather}
\mathscr{A}_0(r), \quad \mathscr{A}_1(r;\xi), \quad \mathscr{A}_2(r;\zeta), \quad \mathscr{A}_3(r;\xi;\zeta), \quad \mathscr{A}_3(p;\xi,\zeta,\eta), \quad \mathscr{A}_4(\alpha;\beta;\xi).
\end{gather}
We can relate these algebras to the reflection equation algebra as follows:
\begin{prop}\label{prop: exact comodule algebras via oh}
For $q$ an odd root of unity, every exact comodule algebra over $\bsltwo$ can be constructed as a quotient of a subalgebra of $\ohtwo$.
\end{prop}
We split the proof of this into parts, recalling the definition of each algebra, and then giving a corresponding presentation in terms of $\ohtwo$.\\

Each of the comodule algebras are defined using some combination of generators $G,X,Y,W$. Denote the $\bsltwo$ coaction by $\delta$. In each case, the coaction on these generators is given by
\begin{gather}
\delta(G) = \hk^{\frac{p}{r}}\otimes G, \qquad \delta(X) = (q-q^{-1}) \hk^{-1}E\otimes 1+\hk^{-1}\otimes X, \qquad \delta(Y) = F\otimes 1+ \hk^{-1}\otimes Y \\
\delta(W) =  \left(\alpha(q-q^{-1})\hk^{-1}E+\beta F\right)\otimes 1+\hk^{-1}\otimes W.
\end{gather}
The defining relations of $\ohtwo$ are:
\begin{gather*}
z_{11}z_{22} = z_{22}z_{11}, \quad z_{11}z_{12} = q^{2}z_{12}z_{11}, \quad z_{11}z_{21} = q^{-2}z_{21}z_{11} \\
z_{12}z_{22}-z_{22}z_{12} = (q^2-1)z_{11}z_{12}, \quad z_{21}z_{22}-z_{22}z_{21} = (1-q^2)z_{21}z_{11} \\ 
z_{12}z_{21}-z_{21}z_{12} = (1-q^2)(z_{22}z_{11}-z_{11}^2).
\end{gather*}
We also recall the central elements $\sigma_1,\sigma_2$ of $\ohtwo$, referred to as the \textit{quantum trace} and \textit{quantum determinant} respectively:
\begin{gather*}
\sigma_1 := qz_{11}+q^{-1}z_{22}, \qquad \sigma_2 := Z_{12,12} = z_{22}z_{11}-q^{-2}z_{21}z_{12}.	
\end{gather*}
The generators and central elements can then be seen to satisfy the following
\begin{gather}\label{eqn: ohtwo alt relations}
z_{12}z_{21} = -q^{2}\sigma_2+q\sigma_1 z_{11}-z_{11}^2 \\
z_{21}z_{12} = -q^{2}\sigma_2+q^3\sigma_1 z_{11}-q^4z_{11}^2
\end{gather}

Before showing how to realize the exact comodule algebras, we first need to modify the coaction on $\ohtwo$ so that it will correctly match with the presentation in \cite{NSS}. As we are only considering odd roots of unity here, then following the discussion in relation to Definition \ref{def: alt bsla}, we can view $\bsltwo$ as a quotient of $\bgltwo$ by the relation $K_1K_2=1$, so that we are in fact technically considering $\bslatwo$. Further, the coaction on $\oh$ given in Proposition \ref{prop: oh coaction to ot} actually maps to $\ugl^{\text{cop}}$ instead of $\ugl$. In the case of $N=2$, we can correct this by applying the isomorphism
\[
(\varphi\otimes\id)\circ\Sigma:\oh\otimes\usltwo^{\text{cop}}\rightarrow \usltwo\otimes\oh,\]
\[
\varphi(E)=F, \quad \varphi(F)=E, \quad \varphi(\hk)=\hk^{-1}, \quad \Sigma(x\otimes y) = y\otimes x.\]
Hence combining these notions, and modifying the coaction by both $\varphi$ and the quotient relation $K_1K_2=1$, we then obtain a coaction
\[
\tilde{\beta}:\ohtwo\rightarrow\usltwo\otimes\ohtwo,\]
In terms of the generators, this is given by
\begin{gather}\label{eqn: modified bsltwo coaction}
\tilde{\beta}(z_{11}) = \hk\otimes z_{11}, \quad \tilde{\beta}(z_{12}) = (1-q^2)E\otimes z_{11}+1\otimes z_{12},  \quad \tilde{\beta}(z_{21}) = (q-q^3) \hk F\otimes z_{11}+1\otimes z_{21} \\
\tilde{\beta}(z_{22}) = q^{-1}(1-q^2)^2 FE\otimes z_{11}+(q^{-1}-q)F\otimes z_{12}+(1-q^2)\hk^{-1} E\otimes z_{21}+\hk^{-1}\otimes z_{22}.
\end{gather}

\subsection*{$\mathscr{A}_0(r)$}

The algebra $\mathscr{A}_0(r)$ is generated by $G$ with the relation
\[
G^r=1.\]
In terms of $\ohtwo$, this is realized as a comodule algebra by
\[
z_{11},z_{12},z_{21}\mapsto 0, \qquad z_{22}^p=1, \qquad G^{-1}\mapsto z_{22}^{\frac{p}{r}}.\]
In particular, $\mathscr{A}_0(p)$ can be obtained directly as a quotient of $\ohtwo$, whilst for $r<p$, $\mathscr{A}_0(r)$ must be obtained as a quotient of a subalgebra of $\ohtwo$.

\subsection*{$\mathscr{A}_1(r,\xi)$}

The algebra $\mathscr{A}_1(r,\xi)$ is generated by $G,X$ with relations
\begin{gather}\label{eqn: GX relations}
G^r=1, \qquad X^N=\xi, \qquad GX=q^{\frac{2p}{r}}XG.
\end{gather}
We can realize this by taking the subalgebra of $\ohtwo$ generated by $z_{11}^{\frac{p}{r}}$, $z_{11}^{-1}z_{12}$, and applying the quotient relations
\[
z_{11}^p=1, \qquad z_{12}^p=-\xi.\]
We then obtain $\mathscr{A}_1(r,\xi)$ as a comodule algebra from $\ohtwo$ via
\[
G\mapsto z_{11}^{\frac{p}{r}}, \qquad X\mapsto -q^{-1}z_{11}^{-1}z_{12}.\]

\subsection*{$\mathscr{A}_2(r,\zeta)$}

The algebra $\mathscr{A}_2(r,\zeta)$ is generated by $G,Y$ with relations
\begin{gather}\label{eqn: GY relations}
G^r=1, \quad Y^p=\zeta, \quad GY = q^{\frac{-2p}{r}}YG 
\end{gather}
We can realize this by taking the subalgebra of $\ohtwo$ generated by $z_{11}^{\frac{p}{r}},z_{11}^{-1}z_{21}$, with the quotient relations
\[
z_{11}^p=1, \qquad z_{21}^p=\zeta(q-q^3)^p.\]
We then obtain $\mathscr{A}_2(r;\zeta)$ as a comodule algebra from $\ohtwo$ via
\[
G\mapsto z_{11}^{\frac{p}{r}}, \qquad Y\mapsto (q-q^3)^{-1}z_{11}^{-1}z_{21}.\]

\subsection*{$\mathscr{A}_3(r;\xi;\zeta)$ and $\mathscr{A}_3(p;\xi;\zeta;\eta)$}

The algebra $\mathscr{A}_3(r;\xi;\zeta)$ is generated by $G,X,Y$, with relations \ref{eqn: GX relations} and \ref{eqn: GY relations}, along with the relation
\begin{gather}\label{eqn: XY relation 1}
XY-q^2YX=1.
\end{gather}
The algebra $\mathscr{A}_3(p;\xi;\zeta;\eta)$ is generated by $G,X,Y$, with relations \ref{eqn: GX relations}, \ref{eqn: GY relations}, and
\begin{gather}\label{eqn: XY relation 2} 
XY-q^2YX=1-\eta G^{-2}.
\end{gather}
Then these algebras are realized from $\ohtwo$ via
\[
G\mapsto z_{11}^{\frac{p}{r}}, \qquad X\mapsto -q^{-1}z_{11}^{-1}z_{12}, \qquad Y\mapsto (q-q^3)^{-1}z_{11}^{-1}z_{21},\]
with quotient relations
\[
z_{11}^p=1, \qquad z_{12}^p=-\xi, \qquad z_{21}^p=\zeta(q-q^3)^p, \qquad \sigma_2=\eta.\]
Where for $A_3(r;\xi;\zeta)$ we take $\eta=0$, so in this case we quotient by $\sigma_2=0$.

We note that in this realization, as we are quotienting to fix the central element $\sigma_2$ by a specific value, this will allow us to write the generator $z_{22}$ in terms of the other elements, i.e.
\[
z_{22} = (\sigma_2+q^{-2}z_{21}z_{12})z_{11}^{-1}.\]
Hence if $r=p$, then these algebras will in fact be realized as quotients of $\ohtwo$, whereas if $r<p$ we instead need to take quotients of subalgebras of $\ohtwo$ to realize them. We note that the $r=p$ algebras in this case can be seen to be closely related to the wee REAs of big cell type.

\subsection*{$\mathscr{A}_4(\alpha;\beta;\xi)$}

The algebra $\mathscr{A}_4(\alpha;\beta;\xi)$ is generated by $W$, with relation $\phi_{\alpha,\beta,\xi}(W)=0$, where $\phi_{\alpha,\beta,\xi}$ is a certain polynomial given in \cite{NSS}[Equation 5.13]. We can realize it as a quotient of $\ohtwo$ via
\begin{gather*}
z_{11}\mapsto 0, \qquad z_{12}\mapsto (q^{-1}-q)^{-1}\beta, \qquad z_{21}\mapsto-q^{-1}\alpha, \qquad \phi_{\alpha,\beta,\xi}(z_{22})=0.
\end{gather*} 
where
\[
W\mapsto z_{22}.\]
The relation $\phi_{\alpha,\beta,\xi}(W)=0$, in terms of the REA, can be thought of as describing the minimal possible algebra generated by the $\bsltwo$ coaction on a certain $\ohtwo$ character. Consider the family of $\ohtwo$ characters
\[ \chi_w(Z) = \left(\begin{array}{cc} 0 & \chi_{12} \\ \chi_{21} & w \end{array}\right) = \left(\begin{array}{cc} 0 & (q^{-1}-q)^{-1}\beta \\ -q^{-1}\alpha & w\end{array}\right),\]
for fixed $\chi_{12},\chi_{21}\in\mc\setminus\{0\}$, and a choice of $w\in\mc$.  Write
\begin{gather}\label{eqn: def of u,v}
u:= \frac{1}{2}(w+\sqrt{w^2+4\chi_{21}\chi_{12}}), \qquad v:= \frac{1}{2}(w-\sqrt{w^2+4\chi_{21}\chi_{12}}),
\end{gather}
so $w=u+v$. Let $\mathcal{V}_2$ denote the two-dimensional irreducible $\bsltwo$ representation. Then for generic choices of $\chi_{12},\chi_{21},w$, and with respect to the coaction \ref{eqn: modified bsltwo coaction}, we have the tensor product decomposition
\begin{gather}\label{eqn: W char tensor product decomposition}
\mathcal{V}_2\otimes\chi_{w}\simeq\chi_{qu+q^{-1}v}\oplus\chi_{q^{-1}u+qv}.
\end{gather}
More generally we have the decomposition
\begin{gather}
	\mathcal{V}_2\otimes\chi_{q^iu+q^{-i}v}\simeq\chi_{q^{i+1}u+q^{-i-1}v}\oplus\chi_{q^{i-1}u+q^{1-i}v}.
\end{gather}
From this, we see that the possible characters that can appear in the decomposition of $\mathcal{V}_2^{\otimes k}\otimes\chi_w$ are
\[
\chi_{q^i u+q^{-i}v}, \quad -k\leq i\leq k.\]
Using $W\mapsto z_{22}$, the relation $\phi_{\alpha,\beta,\xi}(W)=0$, in the form of \cite{NSS}[Equation 5.22], is
\begin{gather}\label{eqn: W relation}
	\phi_{\alpha,\beta,\xi}(z_{22}) = \prod\limits_{i=0}^{p-1}(z_{22}-(q^iu+q^{-i}v)) = 0.
\end{gather}
Hence the relation defining $\mathscr{A}_4(\alpha;\beta;\xi)$ can be seen as describing the characters that can appear in the decomposition of $\mathcal{V}_2^{\otimes k}\otimes\chi_w$.
\begin{defin}
Based on the preceding, we relabel $\mathscr{A}_4(\alpha;\beta;\xi):=\mathscr{A}_{4}^{\chi}(\chi_{12};\chi_{21};w)$, when we view the algebra as coming from an $\ohtwo$ character. Here, $\chi_{12},\chi_{21},w\in\mc$ are chosen to satisfy Relation \ref{eqn: W relation}.
\end{defin}

From our construction of the exact comodule algebras, by considering the cases where the construction is directly obtained as a quotient of $\ohtwo$, we then get the following:
\begin{corr}
	The exact comodule algebras $\mathscr{A}_0(p)$, $\mathscr{A}_3(p;\xi;\zeta;\eta)$, and $\mathscr{A}_4(\alpha;\beta;\xi)$, over $\bsltwo$, are wee reflection equation algebras.  
\end{corr}

\begin{remark}
Based on the results of \cite{Mombelli}, (see also Section \ref{section: Dn module categories}), we expect Proposition \ref{prop: exact comodule algebras via oh} to fail in the even root of unity case, and instead require considering a larger algebra extending $\ohtwo$.
\end{remark}

\section{Semisimplification of Module Categories}\label{section: 5}

Given a non-semisimple tensor category $\mathcal{C}$, we can \textit{semisimplify} it to produce a semisimple tensor category $\widetilde{\mathcal{C}}$. If $\mathcal{M}$ is a module category over $\mathcal{C}$, then we can ask if there a corresponding notion of semisimplification for module categories, to produce a module category $\widetilde{\mathcal{M}}$ over $\widetilde{\mathcal{C}}$. To our knowledge, this has not been defined before, and so we do so now.

As motivation, we first consider the following construction: If $R$ is a ring, $I$ a two-sided $R$-ideal, and $V$ an $R$-module, then
\[
IV:=\{x_1v_1+...+x_kv_k \lvert x_i\in I, v_i\in V\}\]
is an $R$-submodule of $V$. Further, 
\[
V/(IV)\]
will naturally be an $R/I$-module. The semisimplification of module categories can then be thought of as a categorification of this construction. We next recall the definition of semisimplification for tensor categories \cite{BarrettWestbury, EtingofOstrik}.

\begin{defin}
Let $\mathcal{C}$ be a spherical tensor category. A \textbf{tensor ideal} $\mathcal{I}$ of $\mathcal{C}$ is a family of subspaces $\mathcal{I}_{X,Y}\subseteq\Hom(X,Y)$ for objects $X,Y\in\mathcal{C}$ such that for any objects $X,Y,Z,T\in\mathcal{C}$, we have:
\begin{enumerate}
	\item If $\alpha\in \mathcal{I}_{X,Y}$, $\beta\in\Hom(Y,Z)$, then $\beta\circ\alpha\in \mathcal{I}_{X,Z}$.
	\item If $\alpha\in \mathcal{I}_{X,Y}$, $\beta\in\Hom(Z,X)$, then $\alpha\circ\beta\in \mathcal{I}_{Z,Y}$.
	\item If $\alpha\in \mathcal{I}_{X,Y}$, $\beta\in\Hom(Z,T)$, then
	\[ \alpha\otimes \beta\in \mathcal{I}_{X\otimes Z,Y\otimes T}, \quad \text{ and } \quad \beta\otimes\alpha\in \mathcal{I}_{Z\otimes X,T\otimes Y}.\]
\end{enumerate}
\end{defin}
A particularly important example of tensor ideal is the following:
\begin{defin}
Let $\mathcal{C}$ be a spherical tensor category. The \textbf{ideal of negligible morphisms} $\mathcal{N}$ is a tensor ideal defined by $f\in \mathcal{N}_{X,Y}$ if for all $g\in\Hom(Y,X)$ we have $\Tr(f\circ g)=0$.
\end{defin}

\begin{prop}
Let $\mathcal{C}$ be a tensor category and $\mathcal{I}$ a tensor ideal. Then we can construct a quotient tensor category $\mathcal{C}/\mathcal{I}$, with objects the same as $\mathcal{C}$, and morphisms given by
\[
\Hom_{\mathcal{C}/\mathcal{I}}(X,Y):=Hom_{\mathcal{C}}(X,Y)/\mathcal{I}_{X,Y}.\]
In particular, $\mathcal{C}/\mathcal{N}$ is a semisimple tensor category, with objects given by the objects of non-zero dimension in $\mathcal{C}$.	
\end{prop}	
The category $\mathcal{C}/\mathcal{N}$ is called the \textit{semisimplification} of $\mathcal{C}$. We can now introduce a similar notion for module categories:
\begin{defin}
Let $\mathcal{C}$ be a tensor category, $\mathcal{I}$ a tensor ideal of $\mathcal{C}$, and $\mathcal{M}$ a (right) module category over $\mathcal{C}$. The (right) \textbf{module ideal} $\mathcal{J}^{\mathcal{I}}$ of $\mathcal{M}$ with respect to $\mathcal{I}$, is a collection of subspaces $\mathcal{J}^{\mathcal{I}}_{M,N}$ for objects $M,N\in\mathcal{M}$, such that the following conditions hold:
\begin{enumerate}
	\item If $\alpha\in \mathcal{J}^{\mathcal{I}}_{M,N}$ and $\beta\in\Hom(N,P)$ then $\beta\circ\alpha\in\mathcal{J}^{\mathcal{I}}_{M,P}$.
	\item If $\alpha\in\mathcal{J}^{\mathcal{I}}_{M,N}$ and $\beta\in\Hom(P,M)$ then $\alpha\circ\beta\in\mathcal{J}^{\mathcal{I}}_{P,N}$.
	\item If $\alpha\in\Hom(M,N)$ and $\beta\in\mathcal{I}_{X,Y}$, then $\beta\otimes\alpha\in\mathcal{J}^{\mathcal{I}}_{X\otimes M,Y\otimes N}$.
\end{enumerate}
For any objects $M,N,P\in\mathcal{M}$, $X,Y\in\mathcal{C}$.
\end{defin}
\begin{thm}
Let $\mathcal{C}$ be a tensor category, $\mathcal{I}$ a tensor ideal in $\mathcal{C}$, $\mathcal{M}$ a module category over $\mathcal{C}$, and $\mathcal{J}^{\mathcal{I}}$ a module ideal. Then we can construct a module category $\mathcal{M}/\mathcal{J}^{\mathcal{I}}$ over $\mathcal{C}/\mathcal{I}$ as follows: The objects of $\mathcal{M}/\mathcal{J}^{\mathcal{I}}$ are the same as $\mathcal{M}$, and the morphisms are given by
\[
\Hom_{\mathcal{M}/\mathcal{J}^{\mathcal{I}}}(M,N):=\Hom_{\mathcal{M}}(M,N)/\mathcal{J}^{\mathcal{I}}_{M,N}.\]
\end{thm}
\begin{proof}
This follows by construction. The first two conditions for a module ideal ensure that the resulting quotient will make sense as a linear category, whilst the third condition ensures compatibility with the quotient of the corresponding tensor category.
\end{proof}
\begin{defin}
If $\mathcal{C}$ is a tensor category, $\mathcal{N}$ its ideal of negligible morphisms, $\mathcal{M}$ a module ideal over $\mathcal{C}$, and $\mathcal{J}^{\mathcal{N}}$ the associated module ideal, then we call $\mathcal{M}/\mathcal{J}^{\mathcal{N}}$ the \textit{semisimplification} of $\mathcal{M}$ (with respect to $\mathcal{C}$).
\end{defin}
\begin{exmp}
Let $q=e^{\frac{2}{3}i\pi}$, and consider the wee REA $A$ given by 
\[
z_{11}=z_{12}=z_{21}=0, \qquad z_{22}^3=1. \]
This is equivalent to the exact comodule algebra $\mathscr{A}_0(3)$. It is semisimple, with three one-dimensional irreducibles, $\chi_{i}$, $i=0,1,2$, where
\[
\chi_{i}(z_{22})=q^i.\]
The coaction $A\rightarrow \bsltwo\otimes A$ gives $z_{22}\mapsto \hk^{-1}\otimes z_{22}$. Let $\mathcal{P}_3$ denote the three-dimensional projective irreducible $\bsltwo$ representation. Then
\[
\mathcal{P}_3\otimes\chi_{i}\simeq\chi_0\oplus\chi_1\oplus\chi_2\]
for each choice of $i$. We can use this to calculate the semisimplification of $\Rep A$ with respect to the ideal of negligible morphisms $\mathcal{N}$ in $\Rep_0\bsltwo$. As $\mathcal{P}_3$ is projective, $\id_{\mathcal{P}_3}$ is a negligble morphism, i.e. it is in $\mathcal{N}$. Then it follows that 
\[
\id_{\chi_0\oplus\chi_1\oplus\chi_2}\in\mathcal{J}^{\mathcal{N}}.\]
Further, we will have maps
\[
\alpha_i\in\Hom(\chi_i,\chi_0\oplus\chi_1\oplus\chi_2), \qquad \beta_i\in\Hom(\chi_0\oplus\chi_1\oplus\chi_2,\chi_i)\]
such that $\beta_i\circ\alpha_i=\id_{\chi_i}$.
Then composing, we get
\[
\beta_i\circ\id_{\chi_0\oplus\chi_1\oplus\chi_2}\circ\alpha_i=\id_{\chi_i}.\]
Hence it follows that $\id_{\chi_i}\in\mathcal{J}^{\mathcal{N}}$ for each $i$. We can then see from this that the semisimplification of $\Rep A$ with respect to $\Rep_0\bsltwo$ is trivial.
\end{exmp}
This can be seen as an example of a more general property:
\begin{prop}\label{prop: semisimplification of semisimple module categories}
Let $\mathcal{C}$ be a non-semisimple spherical tensor category over $\mc$, $\mathcal{N}$ its ideal of negligible morphisms, and $\mathcal{M}$ a semisimple module category over $\mathcal{C}$. Then the semisimplification of $\mathcal{M}$ with respect to $\mathcal{N}$ is trivial.
\end{prop}
\begin{proof}
Let $M$ be an object in $\mathcal{M}$, and $P$ a projective object in $\mathcal{C}$. Let $\cup\in\Hom(P^*\otimes P,\mc)$. In particular, we will have $\cup\in\mathcal{N}$. Then 
\[
\cup\otimes\id_M\in \Hom(P^*\otimes P\otimes M,M),\]
will be a non-zero map by rigidity, and by construction $\cup\otimes\id_M\in\mathcal{J}^{\mathcal{N}}$. As $\mathcal{M}$ is semisimple, $M$ will therefore be a direct summand of $P^*\otimes P\otimes M$. Hence we can find a non-zero map $f_M\in\Hom(M,P^*\otimes P\otimes M)$ that we can compose to obtain 
\[
(\cup\otimes\id_M)\circ f_M=\id_M\in\mathcal{J}^{\mathcal{N}}.\] 
As this works for any object in $\mathcal{M}$, we see that for all objects $M$ of $\mathcal{M}$, we will have $\id_M\in\mathcal{J}^{\mathcal{N}}$.
\end{proof}
It follows from this proposition that to produce non-trivial examples of semisimplified module categories, then we need to focus on \textit{non-semisimple} module categories.

\subsection{Module categories over $\Rep \usltwo$ fusion categories via semisimplification.}

Our goal now is to show that module categories over $\Rep \usltwo$ fusion categories can be produced via semisimplification of representation categories of $\bsltwo$-comodule algebras. Our focus will be on explicitly constructing the $D_n$ and $T_n$ module categories by semisimplification. We note that the $E_n$ module categories appear to be more difficult to construct, and so we leave for future work.

We first recall the classification of indecomposable module categories over $\Rep \usltwo$ fusion categories, following \cite{EtingofOstrik2} for example. The module categories are typically described in terms of graphs. Let $V$ be the two-dimensional irreducible $\usltwo$ representation. Then to associate a graph to a module category, the vertices of the graph will be labelled by simple objects of the module category. The edges of the graph will encode the fusion rules with respect to $V$. There will be an edge from vertex $M_1$ to vertex $M_2$ if $M_2$ appears in the decomposition of $M_1\otimes V$. Every fusion category is a module category over itself, and in the case of $\Rep \usltwo$ this encodes as the graph $A_{\tp-1}$, where the subscript denotes the number of vertices. Apart from this case, the non-trivial indecomposable module categories are as follows:
\begin{itemize}
\item If $p$ is an odd root of unity, the only non-trivial module category is $T_{\frac{p-1}{2}}$.
\item If $p$ is an even root of unity, and $\tp$ even, there is a non-trivial module category $D_{\frac{\tp}{2}+1}$.
\item If $\tp=12,18,$ or $30$, there is another module category given by the graphs $E_6, E_7$, and $E_8$ respectively.
\end{itemize}
\begin{figure}[H]
	\centering
	\includegraphics[width=0.7\linewidth]{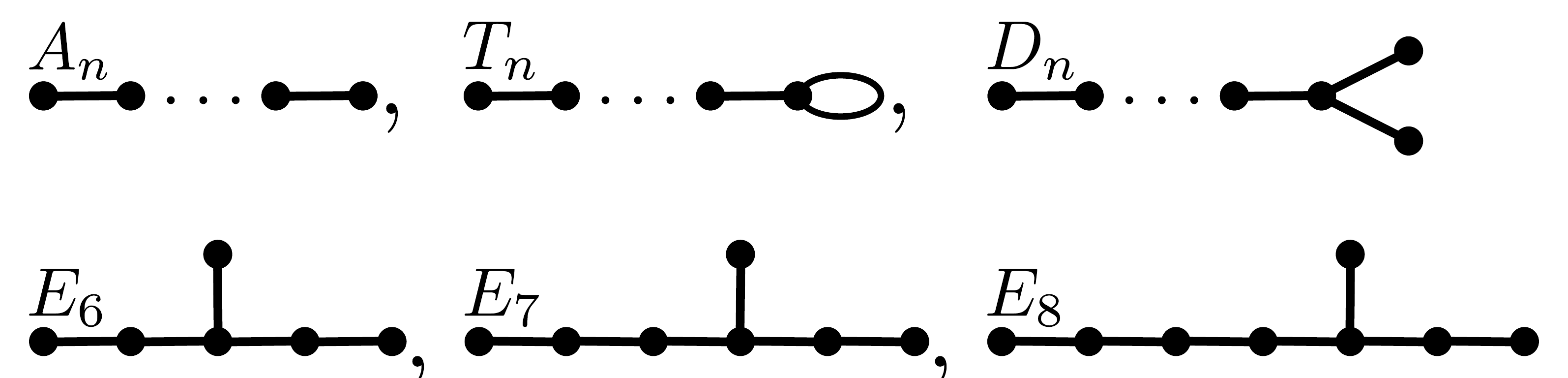}
	\caption{The fusion graphs associated to module categories over $\Rep\usltwo$.}
\end{figure} 
We will first focus on the odd root of unity case, and will show that the $T_n$ module categories can be produced by semisimplification of the exact comodule algebra $\mathscr{A}_{4}^{\chi}(\chi_{12};\chi_{21};w)$. We will then construct the $D_n$ module categories for even roots of unity by considering algebras generalizing $\mathscr{A}_{4}^{\chi}(\chi_{12};\chi_{21};w)$.

Before proceeding to our construction of the module categories, we will first need to recall the structure and fusion rules of the $\bsltwo$ irreducible and projective indecomposable representations:

\begin{prop}
For $q$ an odd root of unity, we denote $\kappa=+$, while for $q$ an even root of unity we denote $\kappa=\pm$. Then the irreducible representations of $\bsltwo$ are $\mathcal{V}^{\kappa}_{n}$ for $1\leq n\leq \tp$. $\mathcal{V}_{n}^{\kappa}$ has basis $\{v_{i}:0\leq i\leq n-1\}$ and action
\[
\hk v_{i} = \kappa q^{n-1-2i}v_{i}, \qquad Ev_{i}=\kappa [i][n-i]v_{i-1}, \qquad Fv_{i} = v_{i+1}.\]
We note that $\mathcal{V}_{\tp}^{\kappa}$ is a projective irreducible representation. The projective indecomposable representations of $\bsltwo$ are $\mathcal{P}_{n}^{\kappa}$ for $1\leq n\leq \tp-1$. $\mathcal{P}_{n}^{\kappa}$ is $2\tp$-dimensional. It has basis 
\[
\{a_i,b_i:0\leq i\leq n-1\}\cup\{x_j,y_j:0\leq j\leq \tp-n-1\}\] 
and action
\begin{gather*}
\hk a_i = \kappa q^{n-1-2i}a_i, \quad \hk b_i = \kappa q^{n-1-2i}b_i, \quad \hk x_j = \kappa q^{p-n-1-2j}x_j, \quad \hk y_j = \kappa q^{p-n-1-2j}y_J, \\
Ea_i = \kappa [i][n-i]a_{i-1}, \quad Eb_i = a_{i-1}+\kappa [i][n-i]b_{i-1}, \quad Eb_0 = x_{\tp-n-1}, \quad Ex_j = \kappa [j][p-n-j]x_{j-1}, \\
Ey_j = \kappa [j][p-n-j]y_{j-1}, \quad Ey_0 = a_{n-1}, \quad Fa_i = a_{i+1}, \\
Fb_i = b_{i+1}, \quad Fb_{n-1} = y_0, \quad Fx_j = x_{j+1}, \quad Fx_{\tp-n-1} = a_0, \quad Fy_j = y_{j+1},
\end{gather*}
where we take $x_{-1}=a_{-1}=a_{n}=y_{\tp-n}=0$. The tensor product decomposition with respect to any combination of irreducible or projective indecomposable representations can be obtained from the following decomposition rules:
\begin{gather*}
	\mathcal{V}^{\kappa_1}_1\otimes\mathcal{V}_{i}^{\kappa_2}\simeq\mathcal{V}_{i}^{\kappa_2}\otimes\mathcal{V}_{1}^{\kappa_1}\simeq \mathcal{V}_i^{\kappa_1 \kappa_2}, \qquad \mathcal{V}_1^{\kappa_1}\otimes\mathcal{P}_i^{\kappa_2}\simeq\mathcal{P}_{i}^{\kappa_2}\otimes\mathcal{V}_{1}^{\kappa_1}\simeq\mathcal{P}_i^{\kappa_1 \kappa_2}, \\
\mathcal{V}_2^{+}\otimes \mathcal{V}_i^{\kappa} \simeq \mathcal{V}_{i}^{\kappa}\otimes\mathcal{V}_2^{+}\simeq \mathcal{V}_{i-1}^{\kappa}\oplus\mathcal{V}_{i+1}^{\kappa}, ~ 2\leq i\leq \tp-1 \\
\mathcal{V}_2^{+}\otimes\mathcal{V}_{\tp}^{\kappa}\simeq\mathcal{V}_{\tp}^{\kappa}\otimes\mathcal{V}_2^{+}\simeq \mathcal{P}_{\tp-1}^{\kappa}, \\
\mathcal{V}_2^{+}\otimes\mathcal{P}_{\tp-1}^{\kappa}\simeq\mathcal{P}_{\tp-1}^{\kappa}\otimes\mathcal{V}_2^{+} \simeq \mathcal{P}_{\tp-2}^{\kappa}\oplus 2\mathcal{V}_{\tp}^{\kappa}, \\
\mathcal{V}_2^{+}\otimes\mathcal{P}^{\kappa}_{i}\simeq\mathcal{P}_{i}^{\kappa}\otimes\mathcal{V}_2^{+}\simeq \mathcal{P}_{i-1}^{\kappa}\oplus\mathcal{P}_{i+1}^{\kappa} ~ 2\leq i\leq \tp-2, \\
\mathcal{V}_2^{+}\otimes \mathcal{P}_1^{\kappa}\simeq\mathcal{P}_{1}^{\kappa}\otimes\mathcal{V}_2^{+}\simeq \mathcal{P}_2^{\kappa}\oplus 2\mathcal{V}_{\tp}^{\kappa\upsilon},
\end{gather*}
where $\upsilon = +$ for odd roots of unity, and $\upsilon = -$ for even roots of unity.
\end{prop}
\begin{proof}
This is well known, see for example \cite{KondoSaito, Suter, Xiao}.
\end{proof}
Consider the tensor category $\Rep_0\bsltwo$ tensor generated by $\bsltwo$ irreducible and projective indecomposable representations. From the above fusion rules, we see it is closed under tensor products, i.e. it will only contain irreducible or projective indecomposable representations as objects. Taking its semisimplification, the resulting fusion category will contain simple objects $\mathcal{V}_i$, $1\leq i\leq p-1$ in the odd root of unity case. We denote this fusion category by $\Rep\usltwo$. In the even root of unity case, the semisimplification will instead be a $\mz_2$ graded version with simple objects $\mathcal{V}_i^{\pm}$, $1\leq i\leq \tp-1$. In this case, the $+$ graded part will form a fusion sub-category. We again denote this fusion sub-category by $\Rep\usltwo$. In what follows, when we consider semisimplification of a module category, it is with respect to the ideal of negligible morphisms in $\Rep_0\bsltwo$.

\subsection{The $T_n$ module categories}

As noted in \cite{NSS}, for certain choices of $\alpha,\beta,\xi$, the algebra $\mathscr{A}_4(\alpha;\beta;\xi)=\mathscr{A}_4^{\chi}(\chi_{12};\chi_{21};w)$ will become non-semisimple. This is the case we now want to consider, and will be the starting point for constructing the various module categories. In terms of Equation \ref{eqn: W relation}, the non-semisimplicity can be seen to occur as a consequence of repeated eigenvalues for characters, i.e. finding for some $0\leq i\neq j<p$ that
\[
q^iu+q^{-i}v = q^ju+q^{-j}v.\]
This can be seen to occur when
\[
v=q^{i+j}u.\]
We then see from this that we can obtain a non-semisimple algebra by choosing $u$ and $v$ accordingly. The case we want to consider is when $u=v=\frac{w}{2}$, which occurs for
\[
\mathscr{A}_{4}^{\chi}(\chi_{12};-\frac{w^2}{4\chi_{12}};w).\]
For this specialization of the algebra, the possible eigenvalues become
\[
(q^i+q^{-i})\frac{w}{2}, \quad 0\leq i\leq \tp.\]
We can summarize the properties of this algebra as follows:
\begin{prop}\label{prop: Summary A4 algebra properties}
For $q$ a root of unity we denote $\mu = 0$ for $q$ odd and $\mu \in\{0,\tp\}$ for $q$ even. Fix $\chi_{12},w\in\mc\setminus\{0\}$, then the algebra $\mathscr{A}_{4}^{\chi}(\chi_{12};-\frac{w^2}{4\chi_{12}};w)$ is a commutative, non-semisimple, and $p$-dimensional quotient of $\ohtwo$ defined by
\[
z_{11}\mapsto 0, \quad z_{12}\mapsto\chi_{12}, \quad z_{21}\mapsto\frac{-w^2}{4\chi_{12}},\] with $z_{22}$ satisfying 
\[
\prod\limits_{i=0}^{p-1}\left(z_{22}-\frac{w}{2}(q^i+q^{-i})\right)=0.\] 
It has projective irreducible representations $\mathscr{I}_{w,\mu}$ given by
\[
\mathscr{I}_{w,\mu}(Z) = \left(\begin{array}{cc} 0 & \chi_{12} \\ -\frac{w^2}{4\chi_{12}} & q^{\mu} w \end{array}\right).\]
It has another $\lfloor\frac{p-1}{2}\rfloor$ projective indecomposable representations, $\mathscr{P}_{w,i}$, $1\leq i\leq \lfloor\frac{p-1}{2}\rfloor$, given by
\begin{gather*}
\mathscr{P}_{w,i}(z_{11}) = \left(\begin{array}{cc} 0 & 0 \\ 0 & 0 \end{array}\right), \quad \mathscr{P}_{w,i}(z_{12}) = \left(\begin{array}{cc} \chi_{12} & 0 \\ 0 & \chi_{12} \end{array}\right), \\
\mathscr{P}_{w,i}(z_{21}) = \left(\begin{array}{cc} \frac{-w^2}{4\chi_{12}} & 0 \\ 0 & \frac{-w^2}{4\chi_{12}} \end{array}\right), \quad \mathscr{P}_{w,i}(z_{22}) = \left(\begin{array}{cc} (q^i+q^{-i})\frac{w}{2} & 0 \\ 1 & (q^i+q^{-i})\frac{w}{2} \end{array}\right).
\end{gather*}
We denote the irreducible quotients of $\mathscr{P}_{w,i}$ by $\mathscr{I}_{w,i}$. The $\bsltwo$ coaction is given by
\[
\tilde{\beta}(z_{22})=(q^{-1}-q)F\otimes\chi_{12}-(1-q^2)\hk^{-1}E\otimes\frac{w^2}{4\chi_{12}}+\hk^{-1}\otimes z_{22}.\] 
The tensor product decompositions with respect to the $2$-dimensional $\bsltwo$ irreducible $\mathcal{V}_2^{+}$ are as follows:
\begin{gather*}
\mathcal{V}_2^{+}\otimes\mathscr{I}_{w,i}\simeq \mathscr{I}_{w,i+1}\oplus\mathscr{I}_{w,i-1} \quad i\neq \mu \\
\mathcal{V}_2^{+}\otimes\mathscr{I}_{w,0}\simeq \mathscr{P}_{w,1} \\
\mathcal{V}_2^{+}\otimes\mathscr{I}_{w,\tp}\simeq \mathscr{P}_{w,\lfloor\frac{p-1}{2}\rfloor}, \quad \text{ for $q$ even}, \\
\mathcal{V}_2^{+}\otimes\mathscr{P}_{w,1} \simeq \mathscr{P}_{w,2}\oplus 2\mathscr{I}_{w,0} \\
\mathcal{V}_2^{+}\otimes\mathscr{P}_{w,1}\simeq 2\mathscr{I}_{w,0}\oplus 2\mathscr{I}_{w,2}, \quad \text{ for $q$ even and $\tp=2$}, \\
\mathcal{V}_2^{+}\otimes\mathscr{P}_{w,i}\simeq \mathscr{P}_{w,i+1}\oplus \mathscr{P}_{w,i-1}, \quad 1<i<\lfloor\frac{p-1}{2}\rfloor \\
\mathcal{V}_2^{+}\otimes\mathscr{P}_{w,\lfloor\frac{p-1}{2}\rfloor} \simeq \mathscr{P}_{w,\lfloor\frac{p-1}{2}\rfloor-1}\oplus \mathscr{P}_{w,\lfloor\frac{p-1}{2}\rfloor}, \quad \text{ for $q$ odd},  \\
\mathcal{V}_2^{+}\otimes\mathscr{P}_{w,\lfloor\frac{p-1}{2}\rfloor} \simeq \mathscr{P}_{w,\lfloor\frac{p-1}{2}\rfloor-1}\oplus 2\mathscr{V}_{w,\tp},	\quad \text{ for $q$ even},
\end{gather*}
where we write $\mathscr{P}_{w,i}=\mathscr{P}_{w,p-i}$, $\mathscr{I}_{w,i}=\mathscr{I}_{w,p-i}$ if $i>\lfloor\frac{p-1}{2}\rfloor$ and $i\neq \mu$.
\end{prop}
\begin{proof}
That the algebra is commutative, non-semisimple, and $p$-dimensional just follows from its construction. The irreducibles and projective indecomposables can be obtained by considering relation \ref{eqn: W relation}. For the tensor product decomposition, we can use the generic decomposition in Equation \ref{eqn: W char tensor product decomposition}, and then specialize for the non-semisimple case. We note that the decompositions involving projective representations must follow as described due to $\mathscr{A}_{4}^{\chi}$ being an exact comodule algebra: We can choose an irreducible $\mathscr{I}_{w,i}$ and $\bsltwo$ projective representation $P$ such that $\mathscr{I}_{w,\mu}$ appears as in the decomposition of $P\otimes\mathscr{I}_{w,i}$. Then as $\mathcal{V}_{2}^{+}\otimes P$ is always projective, we see (by exactness) that $(\mathcal{V}_{2}^{+})^{\otimes k}\otimes \mathscr{I}_{w,\mu}$ can only contain projective representations in its decomposition.
\end{proof}
In fact, from the above proposition we see that every $\mathscr{A}_{4}^{\chi}(\chi_{12};\frac{-w^2}{4\chi_{12}};w)$ indecomposable representation must be one of the irreducible or projective representations listed above. From the tensor product decompositions in the proposition, we can now compute the semisimplification of $\Rep\mathscr{A}_{4}^{\chi}(\chi_{12};\frac{-w^2}{4\chi_{12}};w)$:
\begin{prop}
For $q$ an odd root of unity, the semisimplification of $\Rep\mathscr{A}_{4}^{\chi}(\chi_{12};\frac{-w^2}{4\chi_{12}};w)$ is a non-trivial module category over the $\Rep\usltwo$ fusion category. It has $\frac{p-1}{2}$ simple objects, and its fusion graph is $T_{\frac{p-1}{2}}$.
\end{prop}
\begin{proof}
As $\mathscr{A}_{4}^{\chi}(\chi_{12};\frac{-w^2}{4\chi_{12}};w)$ is an exact comodule algebra, if $W$ is an $\mathscr{A}_{4}^{\chi}(\chi_{12};\frac{-w^2}{4\chi_{12}};w)$ representation, and $P$ a projective $\bsltwo$ representation, then $P\otimes W$ will be a projective representation in $\Rep \mathscr{A}_{4}^{\chi}(\chi_{12};\frac{-w^2}{4\chi_{12}};w)$. Hence as $\id_P\in\mathcal{N}$, any object appearing as a direct summand of the decomposition of $P\otimes W$ will be in the module ideal $\mathcal{J}^{N}$. From the tensor product decompositions in Proposition \ref{prop: Summary A4 algebra properties} we see that every projective representation appears as a direct summand, and so the resulting semisimplification will contain only the non-projective representations as objects. The resulting semisimplified category can then be seen to have fusion rules given by the tadpole graph $T_{\frac{p-1}{2}}$.
\end{proof}
We note in the even root of unity case, the semsimplification of $\Rep \mathscr{A}_{4}^{\chi}(\chi_{12};\frac{-w^2}{4\chi_{12}};w)$ just gives the $A_{\tp-1}$ module category.
\begin{figure}[H]
	\centering
	\includegraphics[width=0.9\linewidth]{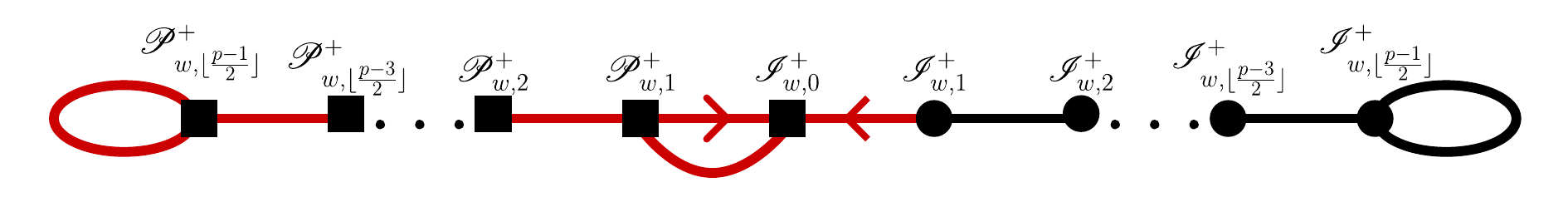}
	\caption{The fusion graph associated to $\Rep\mathscr{A}_4^{\chi}(\chi_{12};\frac{-2^2}{4\chi_{12}};w)$ for $q$ an odd root of unity. The square vertices denote projective objects, and the black edges are those that remain after semisimplification.}
\end{figure} 
\begin{figure}[H]
	\centering
	\includegraphics[width=0.7\linewidth]{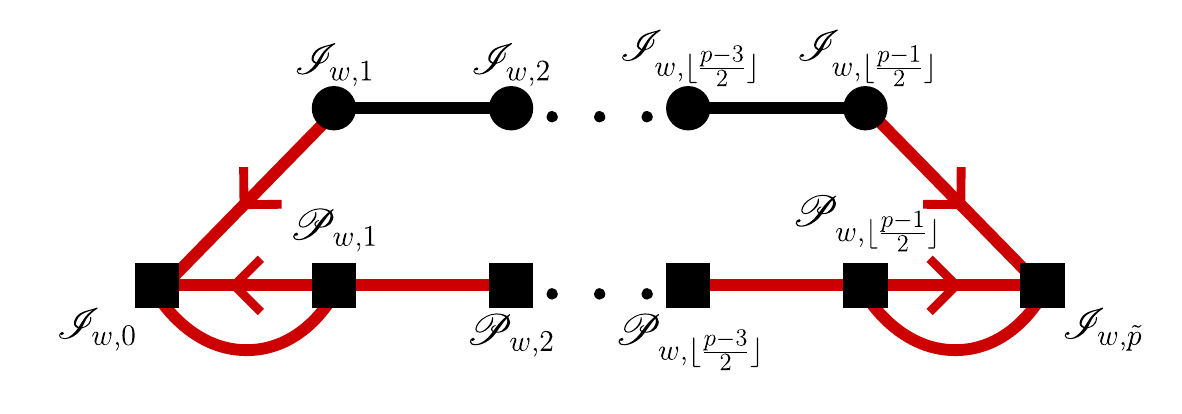}
	\caption{The fusion graph associated to $\Rep\mathscr{A}_4^{\chi}(\chi_{12};\frac{-2^2}{4\chi_{12}};w)$ for $q$ an even root of unity.}
\end{figure}

\subsection{The $D_n$ module categories}\label{section: Dn module categories}

For the even root of unity case, we first note that by the classification of module categories over $\usltwo$ fusion categories, we know that we only need to consider $q$ such that $\tp$ is even. To produce the $D_n$ module categories in this case, we again consider the algebra $\mathscr{A}_{4}^{\chi}(\chi_{12};-\frac{w^2}{4\chi_{12}};w)$ from Proposition \ref{prop: Summary A4 algebra properties}. However, this algebra no longer produces the module categories directly. Instead, we need to consider an extension of this algebra. Motivated by the algebra $\mathcal{A}'(\xi,\mu)$ in \cite{Mombelli}[Section 8.4], we introduce the following:

\begin{defin}
Let $q$ be an even root of unity and $\tp$ even. We define the $\bsltwo$ comodule algebra $\mathscr{B}^{\chi}(\chi_{12};\chi_{21};w)$ to be the extension of the algebra $\mathscr{A}_4^{\chi}(\chi_{12};\chi_{21};w)$ by the extra generator $x$, with relations
\[
xz_{22} = -z_{22}x, \qquad x^4=1, \qquad \delta(x) =  \hk^{\frac{\tp}{2}}\otimes x.\]
\end{defin}
Again for our purposes, we want to consider a case where this algebra is non-semisimple, and so will only consider the specialization $\mathscr{B}^{\chi}(\chi_{12};\frac{-w^2}{4\chi_{12}};w)$. From Proposition \ref{prop: Summary A4 algebra properties}, we can obtain the representations of $\mathscr{B}^{\chi}$:
\begin{prop}
Let $\chi_{12},w\in\mc\setminus\{0\}$ and $q$ be an even root of unity with $\tp$ even. The algebra $\mathscr{B}^{\chi}(\chi_{12};\frac{-w^2}{4\chi_{12}};w)$ is noncommutative, non-semisimple, and $4p$-dimensional. 
It has $\frac{p}{2}+4$ irreducibles, which we denote by $\mathscr{J}_{w,i}^{\pm}$, $0\leq i< \frac{p}{4}$, and $\mathscr{J}_{w,\frac{p}{4}}^{\mu}$, $\mu \in \{\pm 1, \pm i\}$. In particular, $\mathscr{J}_{w,0}^{\pm}$ are the projective irreducible representations. The irreducibles are given by
\begin{gather*}
\mathscr{J}_{w,i}^{\pm}(z_{22}) = \left(\begin{array}{cc} \frac{w}{2}(q^i+q^{-i}) & 0 \\ 0 & \frac{-w}{2}(q^i+q^{-i}) \end{array}\right), \qquad \mathscr{J}_{w,i}^{\pm}(x) = \left(\begin{array}{cc} 0 & 1 \\ \pm 1 & 0 \end{array}\right), \\
\mathscr{J}_{w,\frac{p}{4}}^{\kappa}(z_{22}) = (0), \qquad \mathscr{J}_{w,\frac{p}{4}}^{\mu}(x) = (\mu).
\end{gather*}
We denote the corresponding projective indecomposables by $\mathscr{Q}_{w,i}^{\pm}$, $0< i<\frac{p}{4}$, and $\mathscr{Q}_{w,\frac{p}{4}}^{\mu}$. They are given by
\begin{gather*}
\mathscr{Q}^{\pm}_{w,i}(z_{22}) = \left(\begin{array}{cccc} \frac{w}{2}(q^i+q^{-i}) & 0 & 0 & 0 \\ 1 & \frac{w}{2}(q^i+q^{-i}) & 0 & 0 \\ 0 & 0 & \frac{-w}{2}(q^i+q^{-i}) & 0 \\ 0 & 0 & 1 & \frac{-w}{2}(q^i+q^{-i}) \end{array}\right), \qquad \mathscr{Q}^{\pm}_{w,i}(x) = \left(\begin{array}{cccc} 0 & 0 & 1 & 0 \\ 0 & 0 & 0 & -1 \\ \pm 1 & 0 & 0 & 0 \\ 0 & \mp 1 & 0 & 0 \end{array}\right), \\
\mathscr{Q}_{w,\frac{p}{4}}^{\mu}(z_{22}) = \left(\begin{array}{cc} 0 & 0 \\ 1 & 0 \end{array}\right), \qquad \mathscr{Q}_{w,\frac{p}{4}}^{\mu}(x) = \left(\begin{array}{cc} \mu & 0 \\ 0 & -\mu \end{array}\right). 
\end{gather*}
\end{prop}
\begin{proof}
By construction the algebra is non-commutative and non-semisimple. The classification of its irreducibles and projective indecomposables follows from Proposition \ref{prop: Summary A4 algebra properties}, combined with considering the action of the generator $x$.
\end{proof}
\begin{prop}
The tensor product decompositions of irreducible $\mathscr{B}^{\chi}(\chi_{12};\frac{-w^2}{4\chi_{12}};w)$ representations, with respect to the $\bsltwo$ irreducible $\mathcal{V}_{2}^{+}$, are given as follows:
\begin{gather*}
\mathcal{V}_{2}^{+}\otimes\mathscr{J}_{w,0}^{\pm} \simeq \mathscr{Q}_{w,1}^{\mp}, \\
\mathcal{V}_{2}^{+}\otimes\mathscr{J}_{w,i}^{\pm} \simeq \mathscr{J}_{w,i-1}^{\mp}\oplus\mathscr{J}_{w,i+1}^{\mp}, ~ 0< i<\frac{p}{4}-1, \\
\mathcal{V}_{2}^{+}\otimes\mathscr{J}_{w,\frac{p}{4}-1}^{+} \simeq \mathscr{J}_{w,i-1}^{-}\oplus\mathscr{J}_{w,\frac{p}{4}}^{i}\oplus \mathscr{J}_{w,\frac{p}{4}}^{-i}, \qquad \mathcal{V}_{2}^{+}\otimes\mathscr{J}_{w,\frac{p}{4}-1}^{-} \simeq \mathscr{J}_{w,i-1}^{+}\oplus\mathscr{J}_{w,\frac{p}{4}}^{1}\oplus \mathscr{J}_{w,\frac{p}{4}}^{-1}, \\
\mathcal{V}_{2}^{+}\otimes\mathscr{J}_{w,\frac{p}{4}}^{\pm 1} \simeq \mathscr{J}_{w,\frac{p}{4}-1}^{-}, \qquad \mathcal{V}_{2}^{+}\otimes\mathscr{J}_{w,\frac{p}{4}}^{\pm i} \simeq \mathscr{J}_{w,\frac{p}{4}-1}^{+}.
\end{gather*}
The tensor product decompositions of projective indecomposable $\mathscr{B}^{\chi}(\chi_{12};\frac{-w^2}{4\chi_{12}};w)$ representations, with respect to the $\bsltwo$ irreducible $\mathcal{V}_{2}^{+}$, are given as follows:
\begin{gather*}
\mathcal{V}_{2}^{+}\otimes\mathscr{Q}_{w,1}^{\pm} \simeq \mathscr{Q}_{w,2}^{\mp}\oplus 2\mathscr{J}_{w,0}^{\mp} \\
\mathcal{V}_{2}^{+}\otimes\mathscr{Q}_{w,i}^{\pm} \simeq \mathscr{Q}_{w,i-1}^{\mp}\oplus\mathscr{Q}_{w,i+1}^{\mp}, ~ 0<i<\frac{p}{4}-1, \\
\mathcal{V}_{2}^{+}\otimes\mathscr{Q}_{w,\frac{p}{4}-1}^{+} \simeq \mathscr{Q}_{w,\frac{p}{4}-2}^{-}\oplus\mathscr{Q}_{w,\frac{p}{4}}^{i}\oplus\mathscr{Q}_{w,\frac{p}{4}}^{-i}, \qquad \mathcal{V}_{2}^{+}\otimes\mathscr{Q}_{w,\frac{p}{4}-1}^{-} \simeq \mathscr{Q}_{w,\frac{p}{4}-2}^{+}\oplus\mathscr{Q}_{w,\frac{p}{4}}^{1}\oplus\mathscr{Q}_{w,\frac{p}{4}}^{-1}, \\
\mathcal{V}_{2}^{+}\otimes\mathscr{Q}_{w,\frac{p}{4}}^{\pm 1} \simeq \mathscr{Q}_{w,\frac{p}{4}-1}^{-}, \qquad \mathcal{V}_{2}^{+}\otimes\mathscr{Q}_{w,\frac{p}{4}}^{\pm i} \simeq \mathscr{Q}_{w,\frac{p}{4}-1}^{+}.
\end{gather*}
For $\tp=2$ we instead have the following decompositions
\[
\mathcal{V}_{2}^{+}\otimes\mathscr{J}_{0}^{+}\simeq\mathscr{Q}^{i}_{1}\oplus\mathscr{Q}^{-i}_{1}, \quad \mathcal{V}_{2}^{+}\otimes\mathscr{J}_{0}^{-}\simeq\mathscr{Q}^{1}_{1}\oplus\mathscr{Q}^{-1}_{1}, \quad
\mathcal{V}_{2}^{+}\otimes\mathscr{Q}_{1}^{\pm 1}\simeq 2\mathscr{J}^{-}_{0}, \quad \mathcal{V}_{2}^{+}\otimes\mathscr{Q}_{1}^{\pm i}\simeq 2\mathscr{J}^{+}_{0}.\]
Finally, with respect to the one-dimensional $\bsltwo$ representation $\mathcal{V}_{1}^{-}$, we have
\begin{gather*}
\mathcal{V}_{1}^{-}\otimes\mathscr{J}_{w,i}^{\kappa} \simeq \mathscr{J}_{w,i}^{-\kappa}, \qquad \mathcal{V}_{1}^{-}\otimes\mathscr{Q}_{w,i}^{\kappa} \simeq \mathscr{Q}_{w,i}^{-\kappa}, 
\end{gather*}
for any valid choice of $i$ and $\kappa$.
\end{prop}
\begin{proof}
This follows from the fusion rules in Proposition \ref{prop: Summary A4 algebra properties}, after accounting for the extra generator $x$.
\end{proof}
Using the above tensor product decompositions, we can now state how to construct the $D_n$ module categories via semisimplification:
\begin{thm}
For $q$ an even root of unity and $\tp$ even, consider the $\mathscr{B}^{\chi}(\chi_{12};\frac{-w^2}{4\chi_{12}};w)$ representation $\mathscr{J}_{w,1}^{+}$. Then it generates a module category $\mathcal{M}_{\mathscr{J}}$ over $\Rep_0\bsltwo$, such that the semisimplification of $\mathcal{M}_{\mathscr{J}}$ is a module category with fusion graph $D_{\frac{\tp}{2}+1}$ over the $\Rep \usltwo$ fusion category.
\end{thm}
We note that the $\mathscr{B}^{\chi}(\chi_{12};\frac{-w^2}{4\chi_{12}};w)$ representation $\mathscr{J}_{w,1}^{-}$ will also generate a copy of the $D_{\frac{\tp}{2}+1}$ module category via semisimplification. The fusion graphs for both representations are given as follows:
\begin{figure}[H]
	\centering
	\includegraphics[width=0.9\linewidth]{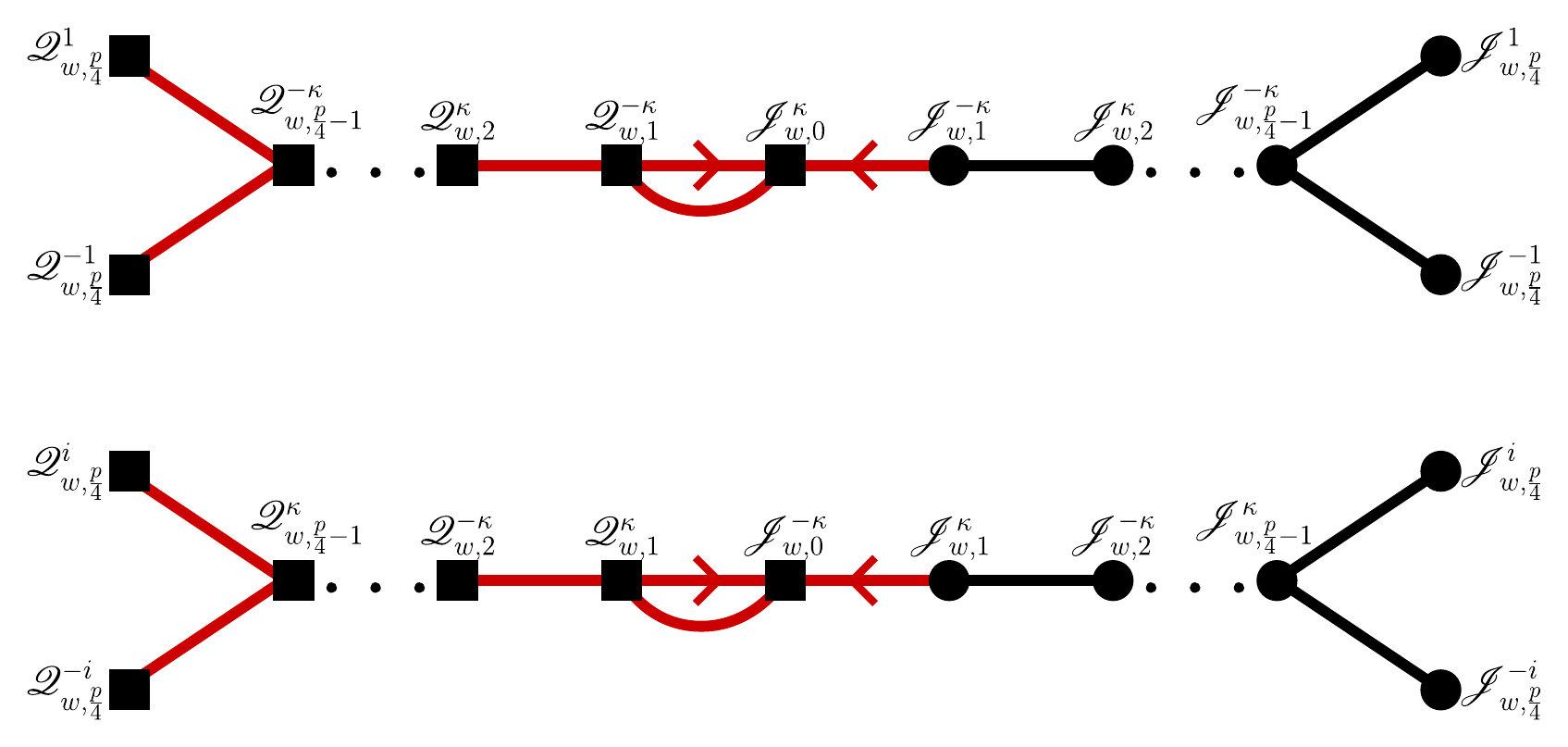}
	\caption{The fusion graphs associated to $\Rep\mathscr{B}^{\chi}(\chi_{12};\frac{-2^2}{4\chi_{12}};w)$ for $q$ an even root of unity and $\tp$ even. Here we denote $\kappa=+$ if $\tp$ is divisible by $4$, and $\kappa=-$ otherwise.}
\end{figure}

\bibliography{ref}
\bibliographystyle{abbrv}

\Addresses

\end{document}